\documentclass[preprint,12pt]{elsarticle}
\usepackage{amsmath,amssymb,latexsym}
\usepackage{amsthm}
\usepackage{bm}
\usepackage[thinlines]{easymat}
\usepackage{overpic} 
\newtheorem{theorem}{Theorem} 
\newtheorem{proposition}{Proposition}
\newtheorem{corollary}{Corollary} 
\newtheorem{lemma}{Lemma}
\theoremstyle{definition}

\theoremstyle{remark}
\newtheorem{remark}{Remark}

\newcommand{\res}{\mathop{\rm res}}

\newcommand{\supp}{\mathop{\rm supp}}

\DeclareMathOperator{\Ai}{Ai}

\newcommand{\field}[1]{\mathbb{#1}}

\newcommand{\R}{\field{R}}
 
\newcommand{\N}{\field{N}}
\newcommand{\C}{\field{C}}

\newcommand{\const}{{\rm const}}

\renewcommand{\Re}{\mathop{\rm Re}}
\renewcommand{\Im}{\mathop{\rm Im}}

\def\XXint#1#2#3{{\setbox0=\hbox{$#1{#2#3}{\int}$}
\vcenter{\hbox{$#2#3$}}\kern-.5\wd0}}

\journal{Journal of Mathematical Analysis and Applications}

\begin{document}

\begin{frontmatter}

\title{Quadratic differentials and asymptotics of Laguerre polynomials with varying complex parameters}

\author[address1]{M.~J.~Atia}
\ead{jalel.atia@gmail.com}
\author[address2,address3]{A. Mart\'{\i}nez--Finkelshtein\corref{cor1}}
\ead{andrei@ual.es}
\cortext[cor1]{Corresponding author.}
\author[address2]{P.~Mart\'{\i}nez--Gonz\'alez}
\ead{pmartine@ual.es}
\author[address4]{F.~Thabet}
\ead{faouzithabet@yahoo.fr}
\address[address1]{Department of Mathematics, 
Facult\'e des sciences de Gab\'es
Gab\'es, TUNISIA}
\address[address2]{Department of Mathematics,
University of Almer\'{\i}a, SPAIN}
\address[address3]{Instituto Carlos I de F\'{\i}sica Te\'{o}rica y Computacional,
Granada University, SPAIN}
\address[address4]{ISSAT Gab\'es, 
Gab\'es, TUNISIA}

\begin{abstract}

In this paper we study the asymptotics (as $n\to \infty$) of the sequences of Laguerre polynomials with varying complex parameters
$\alpha$ depending on the degree $n$. More precisely, we assume that
$\alpha_n = n A_n, $ and $ \lim_n A_n=A \in \C$. 
This study has been carried out previously only for $\alpha_n\in \R$, but complex values of $A$ introduce an asymmetry that makes the problem more difficult.

The main ingredient of the asymptotic analysis is the right choice of the contour of orthogonality, which requires the analysis of the global structure of trajectories of an associated quadratic differential on the complex plane, which may have an independent interest. 

While the weak asymptotics is obtained by reduction to the theorem of Gonchar--Rakhmanov--Stahl, the strong asymptotic results are derived via the non-commutative steepest descent analysis based on the Riemann-Hilbert characterization of the 
Laguerre polynomials. 
\end{abstract}


\end{frontmatter}

\section{Introduction} \label{sec:intro}

One of the motivations of this paper is the asymptotic analysis of the generalized Laguerre polynomials, denoted by $L_n^{(\alpha)}$, with complex varying parameters, whose definition and  properties can be found for instance in Chapter V of Szeg\H{o}'s classic
memoir \cite{szego:1975}. They can be given explicitly by
\begin{equation}\label{explLag}
L_n^{(\alpha)} (z)=\sum_{k=0}^n
\binom{n+\alpha}{n-k}\frac{(-z)^{k}}{k!},
\end{equation}
or, equivalently, by the well-known Rodrigues formula
\begin{equation}\label{RodrLag}
L_n^{(\alpha)} (z)=\frac{(-1)^n}{n!}\, z^{-\alpha} e^{z}\left(
\frac{d}{dz} \right)^n \left[ z^{n+\alpha} e^{-z}\right]\,.
\end{equation}
Expressions \eqref{explLag} and \eqref{RodrLag} make sense for  complex values of the parameter $\alpha$, showing that $L_n^{(\alpha)}$ depend analytically on $\alpha$. When indeterminacy occurs in evaluating the coefficients in \eqref{explLag}, we understand them in the sense of their analytic continuation with respect to $\alpha$. With this convention we see that
$$
L_n^{(\alpha)} (z)=\frac{(-1)^n}{ n!} \, z^n + \text{ lower degree terms},
$$
so that  $\deg L_n^{(\alpha)} = n$ for all $\alpha \in \C$. Moreover, for any   $\alpha \in \C$, $L_n^{(\alpha)}$ is the unique (up to a multiplicative
constant) polynomial solution of the differential equation
\begin{equation}\label{difLag}
  z y''(z) +(\alpha + 1 - z) y'(z)+n y(z)=0,
\end{equation}
which shows that every zero of $L_n^{(\alpha)}$ different from $z=0$
must be simple.  In fact, multiple zeros (at the origin) can appear if and only if $\alpha \in
\{-1, -2, \dots, -n\}$. In this case the reduction formula
\begin{equation*}
L_{n}^{(-k)}(z) =
(-z)^{k}\frac{(n-k)!}{n!}\, L_{n-k}^{(k)}(z)\,,
\end{equation*}
shows that $z=0$ is a zero of $L_{n}^{(-k)}(z)$ of multiplicity $k$, see again \cite{szego:1975} for details.

Orthogonality conditions satisfied by the Laguerre polynomials can be easily derived from \eqref{RodrLag} iterating integration by parts, see e.g.~\cite{MR1858305}; we reproduce the arguments in Section \ref{sec:orthogonality} for the sake of completeness. The weight of orthogonality is the algebraic function $ z^{\alpha} e^{-z}$ and the integration goes along a contour in the complex plane.
The classical situation is $\alpha>-1$, in which case the orthogonality of
$L_n^{(\alpha)}(x)$   reduces to
\begin{equation*} 
    \int_0^{\infty} L_n^{(\alpha)}(x) L_m^{(\alpha)}(x)
    x^{\alpha} e^{-x}\, dt = 0, \qquad
    \mbox{ if } n \neq m.
\end{equation*}
As a consequence, for $\alpha > -1$ zeros of $L_n^{(\alpha)}(x)$ are positive and
simple.

In this paper we study sequences of Laguerre polynomials with, in general, complex parameters
$\alpha$ depending on the degree $n$. More precisely, we assume that
\begin{equation}\label{limits}
\alpha_n = n A_n, \quad \text{and}\quad \lim_n A_n=A \in \C.
\end{equation}
With these hypotheses, but only for \emph{real} parameters $\alpha_n$, sequences $L_n^{(\alpha_n)}$ were studied, in particular, in \cite{MR1858305, MO} (weak asymptotics) and in \cite{Kuijlaars/McLaughlin:01, Kuijlaars/Mclaughlin:04} (strong asymptotics). This paper is a natural continuation of this study, although complex values of $A$ introduce an asymmetry that makes the problem more difficult. The main ingredient of the asymptotic analysis is the right choice of the contour of orthogonality, which is related to the trajectories of an associated quadratic differential. Precisely the description of the structure of these trajectories (Section \ref{sec:trajectories}) constitutes the core of our contribution.

It is known that under assumptions \eqref{limits} we need to perform a linear scaling in the variable in order to fix the geometry of the problem.
Thus, we will study the sequence
\begin{equation}\label{sequencex}
p_n(z)=L_n^{(\alpha_n)}(n z)=\frac{(-n)^n}{n!}z^n + \text{ lower degree terms} .
\end{equation}
The zeros of  $p_n$ cluster along certain curves
in the complex plane, corresponding to
trajectories  (known also as Stokes lines) of a quadratic differential, depending on a parameter $A$, see Figure~\ref{fig:CerosTray}. This is not surprising: as it follows from the pioneering works of  Stahl \cite{Stahl:86}, and later of Gonchar and Rakhmanov \cite{Gonchar:84, Gonchar:87}, the support of the limiting zero-counting measure of such polynomials is a set of analytic curves exhibiting the so-called $S$-property. They can also be characterized as trajectories of a certain quadratic differential on the Riemann surface. The explicit expression of the quadratic differential associated to polynomials $p_n$ can be easily derived from the differential equation \eqref{difLag}, see \cite{MR1858305}  for details.

\begin{figure}
\centerline{\includegraphics[scale=0.5]{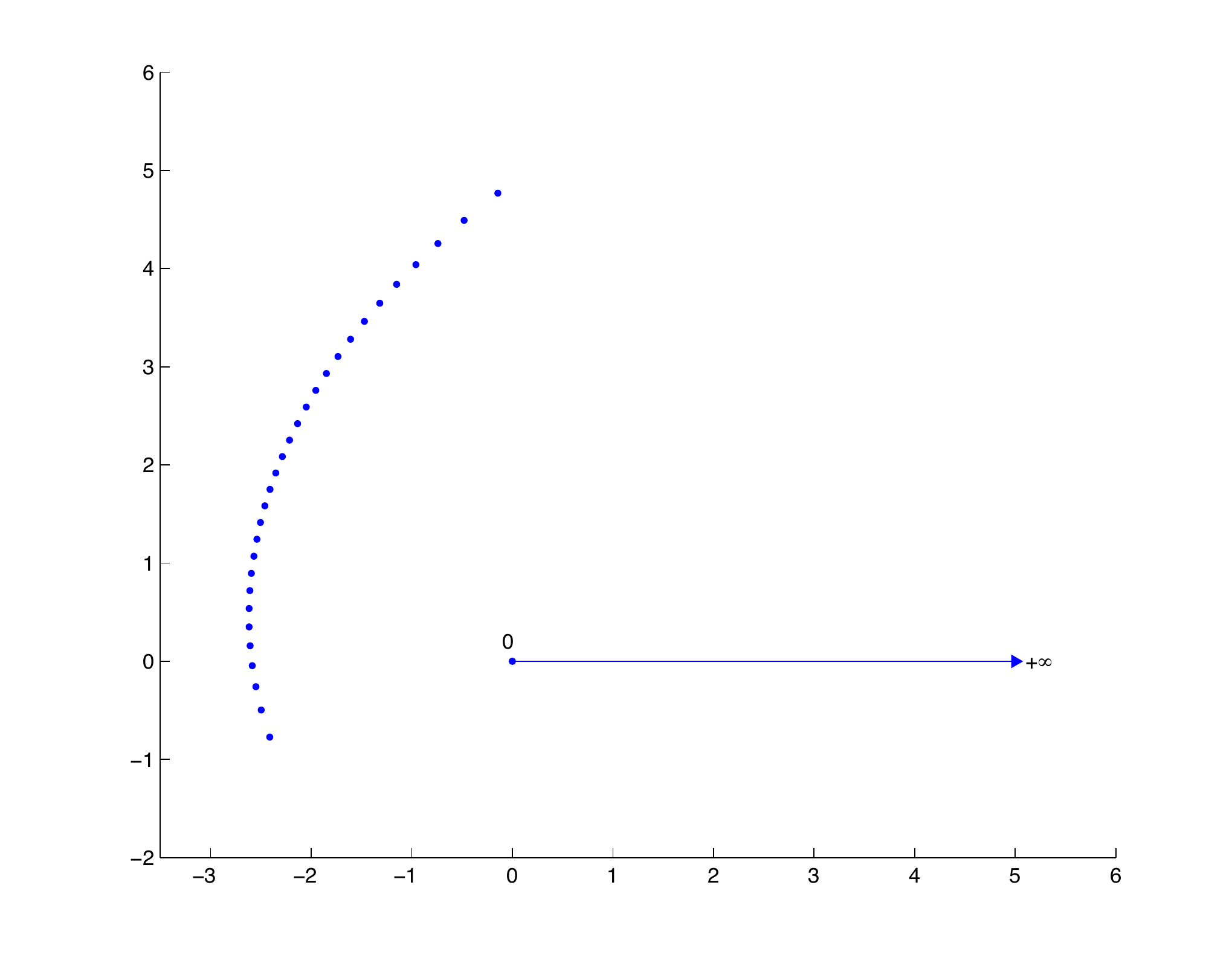}}
\caption{Zeros of $L_n^{(A n)}(nz)$ for
$A = -3+2i$   and $n=30$.
}
\label{fig:CerosTray}
\end{figure}

It turns out that for $A < -1$, the support of the limiting zero-counting measure for $p_n$'s  is a simple analytic and real-symmetric arc, which as $A$ approaches $-1$, closes
itself to form for $A = -1$ the well-known Szeg\H{o} curve \cite{MR1407500, szego:1975}. When $A\geq 0$, the support becomes an interval of the positive semi-axis.  Case $-1 < A < 0$ is special: generically, the support is connected, and consists of a closed loop surrounding the origin together with an interval of the positive semi-axis. However, when $\alpha_n$'s are exponentially close to integers, the support can split into two disjoint components, the closed contour and the interval, see \cite{Kuijlaars/Mclaughlin:04} for details.

In the case $A\notin \R$ the trajectories of these quadratic differentials, and subsequently, the support of the limiting zero-counting measure for $p_n$'s, have not been described; this is done in Section \ref{sec:trajectories}, which as we pointed out, is probably one of the central contributions of this paper.

Being the generalized Laguerre polynomials such a classical object, it is not surprising that their asymptotics has been well studied, using different and complementary approaches based on many characterizations of these polynomials. Many of these results correspond to the case $\alpha_n>-1$, when all their zeros belong to the positive semi-axis, see e.g.~\cite{Bosbach/Gawronski98}. Others study the asymptotics with a \emph{fixed} parameter $\alpha$, such as in \cite{MR2448665}, where additionally many interdisciplinary applications of Laguerre asymptotics are explained.
The Gonchar-Rakhmanov theory was used to find the weak asymptotics (or asymptotic zero distribution) of $p_n$'s for general $\alpha_n\in \R$ satisfying \eqref{limits}, see \cite{MR1858305}; we extend it to complex $\alpha_n$'s in this work, see Section~\ref{sec:results}. The critical case $A=-1$ in this context was analyzed in \cite{MO}, using the extremality of the family of polynomials.

The  non-linear steepest descent method of Deift--Zhou introduced in \cite{MR94d:35143}, and further developed in \cite{MR98b:35155} and \cite{MR96d:34004} (see also \cite{MR2000g:47048}), based on the Riemann--Hilbert characterization of  orthogonality by Fokas, Its,
and Kitaev \cite{Fokas92}, is an extremely powerful technique, rendering exhaustive answers in cases previously intractable. Following this approach, Kuijlaars and McLaughlin \cite{Kuijlaars/McLaughlin:01, Kuijlaars/Mclaughlin:04} found the strong asymptotics in the whole complex plane for the family $\{p_n\}$ and arbitrary values of $A\in \R$. The crucial ingredient of this asymptotic analysis is the choice of an appropriate path of integration on the complex plane, based on the structure of the trajectories of the associated quadratic differential.

Taking advantage of the results of Section \ref{sec:trajectories} we prove the existence and describe such paths of integration, which allows us to carry out the steepest descent analysis in the spirit of Kuijlaars and McLaughlin. As a result, we obtain in Section~\ref{sec:results} the detailed strong asymptotics for the rescaled Laguerre polynomials $p_n$. This also sheds light on one of the open questions mentioned in  \cite{MR2448665}.

The generalized Bessel polynomials $B_n^{(\alpha)}$ can be defined as
\begin{equation*}
B_n^{(\alpha)}(z)=z^n L_n^{(-2n-\alpha+1)}\left(\frac{2}{z}
\right)\,.
\end{equation*}
The asymptotic distribution of their zeros and their strong asymptotics is a straightforward consequence of
Theorems \ref{teoLag1} and \ref{thm:strongA} below, by replacing
$A \mapsto -(A+2)$ and $z \mapsto 2/z$.

\section{Trajectories of a family of quadratic differentials}\label{sec:trajectories}

Let $A\in \C$ be a complex parameter, for which we define the monic polynomials
$$
D(z)=D_A(z)=(z-A)^2-4z =(z-\zeta_+)(z-\zeta_-),
$$
with
\begin{equation}
\label{formulaZeros}
\zeta_\pm = \zeta_\pm(A) =  A+2\pm 2 \sqrt{A+1}=(1\pm \sqrt{A+1})^2.
\end{equation}
Since $D_A$ is real-symmetric with respect to the parameter $A$,  without loss of generality we can assume in what follows that $\Im (A)\geq 0$, and that the square root in \eqref{formulaZeros} stands for its main branch in the closed upper half plane. In what follows, we use  the notation
$$
\C_+=\{z\in \C:\, \Im z>0\}, \quad \C_-=\{z\in \C:\, \Im z<0\},
$$
while as usual, $\R_-$ and $\R_+$ stand for the open positive and negative real semi-axes, respectively.

On the Riemann sphere $\overline \C$ we define the quadratic differential
$$
\varpi_A= -\frac{D(z)}{z^2}\,dz^2,
$$
written in the natural parametrization of the complex plane. Its \emph{horizontal trajectories} (or just trajectories in the future) are the loci of the equation
$$
\Re \int^{z} \frac{\sqrt{D(t)}}{t}\, dt \equiv \const;
$$
the \emph{vertical} or \emph{orthogonal} trajectories are obtained by replacing $\Re$ by $\Im$ in the equation above. The trajectories and the orthogonal trajectories of  $\varpi_A$ produce a transversal foliation of the Riemann sphere $\overline \C$.

In order to study the global structure of these trajectories on the plane we start by observing that $\varpi_A$
has two zeros, $\zeta_\pm$, that are distinct and simple if and only if $A\neq -1$, and a
double pole at the origin if $A\neq 0$, with
$$
\varpi_A=\left(-\frac{A}{z^2} +\mathcal O(z^{-1})\right)  dz^2, \quad z\to 0.
$$
Another pole of $\varpi_A$ is located at infinity and is of order 4; with the parametrization $u=1/z$,
$$
\varpi_A=\left(-\frac{1}{u^4} +\mathcal O(u^{-3})\right)  du^2, \quad u\to 0.
$$
Points from $\overline{\C}\setminus\{0,\zeta_-, \zeta_+ ,\infty\}$ are regular.

The local structure of the trajectories is well known (see e.g.~\cite{MR0096806}, \cite{Pommerenke75}, \cite{Strebel:84}, \cite{MR1929066}). At any regular point trajectories look locally as simple analytic arcs passing through this point, and through every regular point of $\varpi_A$ passes a uniquely determined horizontal and uniquely determined vertical trajectory of $\varpi_A$, that are locally orthogonal at this point \cite[Theorem 5.5]{Strebel:84}.

For $A\neq -1,0$, there are $3$ trajectories emanating from $\zeta_\pm$ under equal angles $2\pi/3$. In the case of the origin, the trajectories have either the radial, the circular or the log-spiral form, depending on the vanishing of the real or imaginary part of $A$, see Figure \ref{fig:localstructure2}.

\begin{figure}[htb]
\centering \begin{tabular}{lll} \hspace{-1.5cm}\mbox{\begin{overpic}[scale=0.47]%
{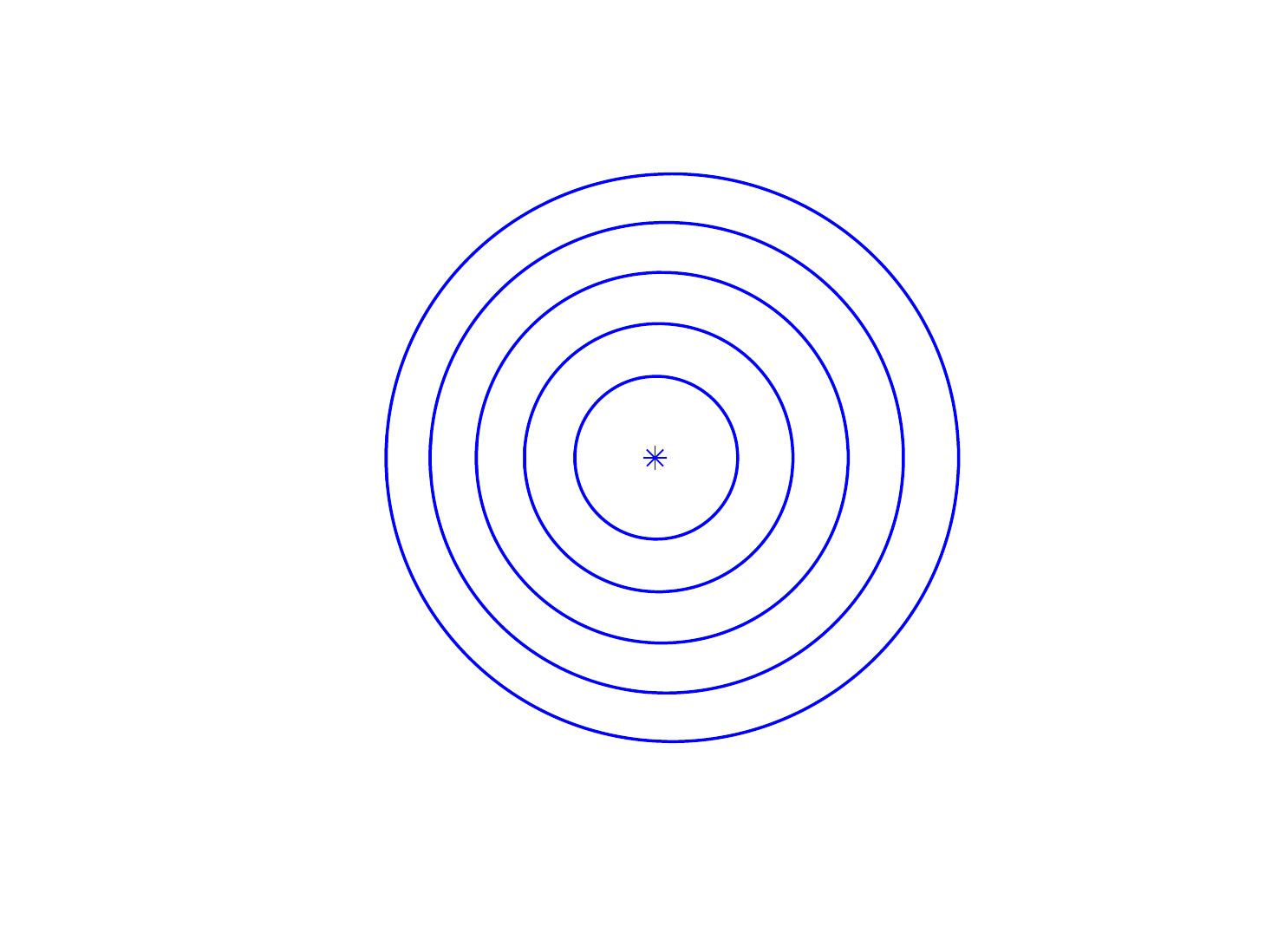}%
\end{overpic}} &
\hspace{-2.0cm}
\mbox{\begin{overpic}[scale=0.45]%
{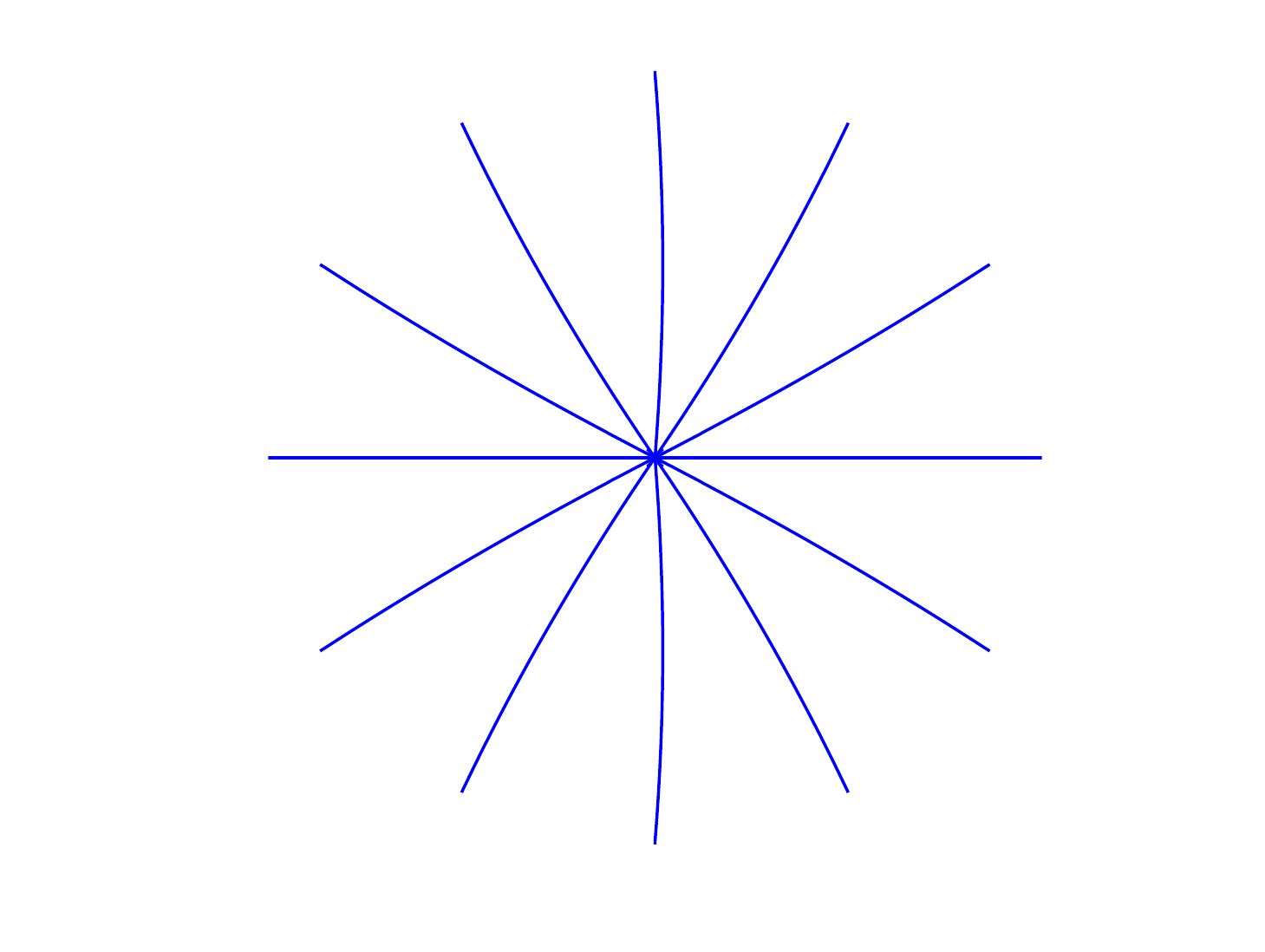}%
\end{overpic}}&
\hspace{-2.0cm}
\mbox{\begin{overpic}[scale=0.45]%
{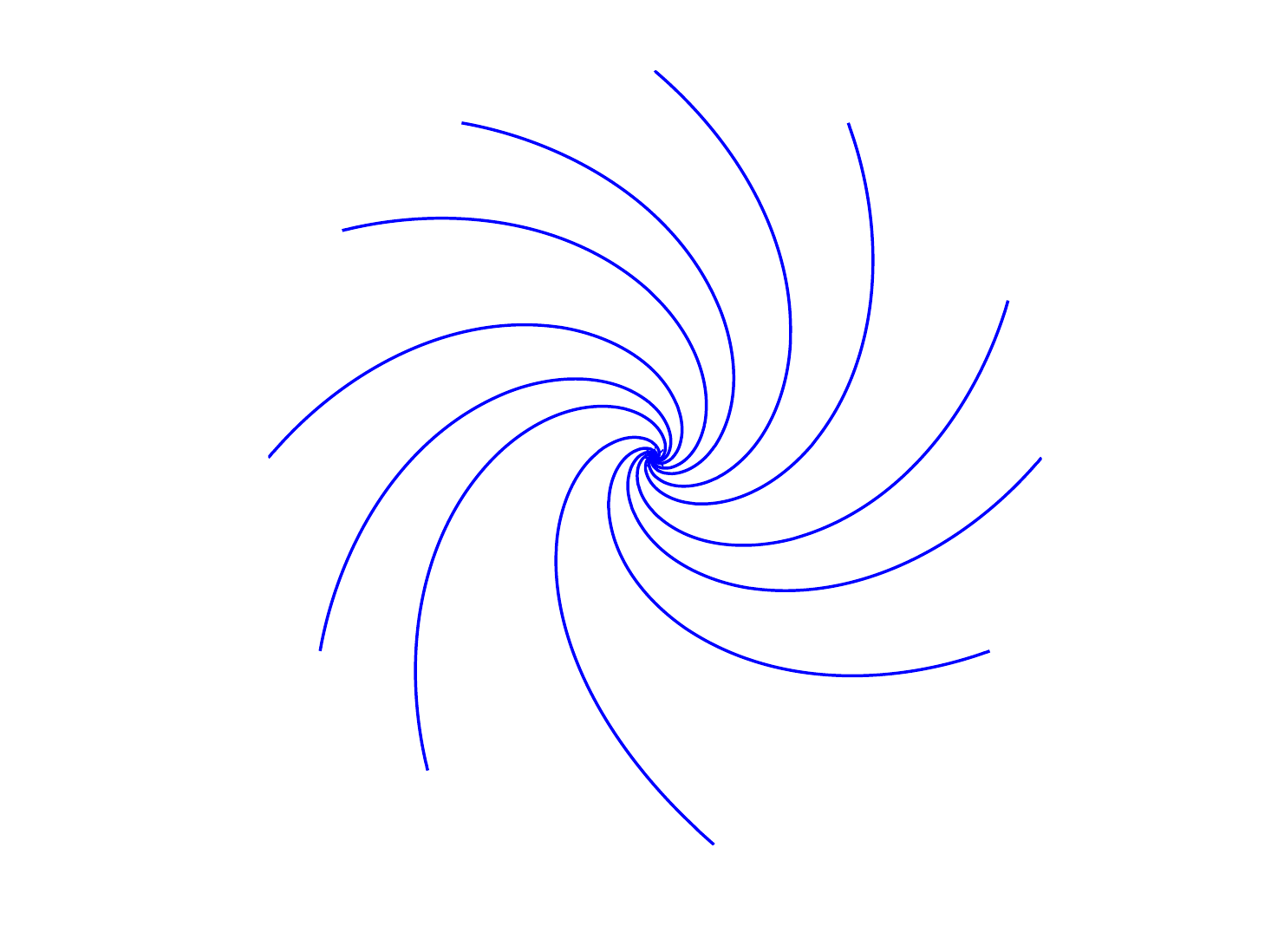}%
\end{overpic}}
\end{tabular}
\caption{The local trajectory structure of $\varpi_A$ near the origin when $\Im A =0$ (left), $\Re A=0$ (center) and in the rest of the cases ($A\neq 0$).}
\label{fig:localstructure2}
\end{figure}

Regarding the behavior at infinity, we infer that  the imaginary axis is the only asymptotic direction of the trajectories of $\varpi_A$; there exists a neighborhood of infinity $D$ such that every trajectory  entering $D$ tends to $\infty$ either in the $+i\infty$ or $-i\infty$ direction, and the two rays of any trajectory which stays in $D$ tend to $\infty$ in the opposite asymptotic directions (\cite[Theorem 7.4]{Strebel:84}).

A trajectory $\gamma$  of $\varpi_A$ starting and ending at $\zeta_\pm$ (if exists) is called \emph{finite critical} or \emph{short}; if it starts at one of the zeros $\zeta_\pm$ but tends either to the origin or to infinity, we call it \emph{infinite critical trajectory} of $\varpi_A$. In a slight abuse of terminology, we say that such an infinite critical trajectory, if it exists,  \emph{joins} the zero with either the origin or the infinity.

The set of both finite and infinite critical trajectories of $\varpi_A$ together with their limit points (critical points of $\varpi_A$) is the \emph{critical graph} $\Gamma_A$ of $ \varpi_A$.

In this section we describe the global structure of the trajectories of $\varpi_A$, essentially determined by the critical graph $\Gamma_A$, as well as of its orthogonal trajectories.
Usually, the main troubles come from  the existence of the so-called recurrent trajectories, whose closure may have a non-zero plane Lebesgue measure. However, since $\varpi_A$ has two poles, Jenkins' Three Pole theorem asserts that it cannot have any recurrent trajectory (see  \cite[Theorem 15.2]{Strebel:84}).

One of the main result  of this section is the following theorem, which collects the properties of the critical graph of $\varpi_A$  (see Figure~\ref{fig:globalstructure}).
\begin{theorem} \label{thm:1}
For any $A\in \C$  there exists a short trajectory $\gamma_A$ of $\varpi_A$, joining $\zeta_-$ and $\zeta_+$. 

If $A\notin \R$, this trajectory is unique, homotopic in the punctured plane $\C\setminus \{0\}$ to a Jordan arc connecting $\zeta_\pm$ in $\C\setminus \R_+$, and it intersects the straight segment, joining $\zeta_-$ and $\zeta_+$, only at its endpoints,  $\zeta_-$ and $\zeta_+$.

Furthermore, for $A\in \C_+$ the structure of the critical graph $\Gamma_A$ of $\varpi_A$ is as follows:
\begin{itemize}
\item the short trajectory $\gamma_A$ of $\varpi_A$, joining $\zeta_-$ and $\zeta_+$;
\item  the  unique infinite critical trajectory $\sigma_{0}$ of $\varpi_A$ emanating from   $\zeta_-$ and diverging to the origin;
\item the  critical trajectory $\sigma_-$, emanating from   $\zeta_-$ and diverging towards $-i\infty$;
\item two  critical trajectories $\sigma_{\uparrow+}$ and $\sigma_{\downarrow+}$, emanating from   $\zeta_+$ and diverging towards $+i\infty$ and $-i\infty$, respectively.
\end{itemize}
$\Gamma_A$ splits $\C$ into three connected domains, two of them of the half-plane type. The domain, bounded by $\sigma_{-}\cup \gamma_A\cup \sigma_{\downarrow+}$, with the inner angle $2\pi/3$ at $\zeta_+$, is a strip domain and contains the origin.
\end{theorem}
In other words, we claim that in the non-real case the critical graph of $\varpi_A$ is made of one short and 4 infinite critical trajectories. The notion of half-plane and strip domains essentially means that
$$
 \int^{z}  \frac{\sqrt{D_A(t)}}{t}\, dt
 $$
is a conformal mapping of this domain onto a vertical half-plane or a vertical strip, respectively. See Proposition~\ref{prop:confmapping} below or \cite[\S 10]{Strebel:84} for details.

\begin{figure}[htb]
\centering \begin{tabular}{lll} \hspace{-0.2cm}\mbox{\begin{overpic}[scale=0.35]%
{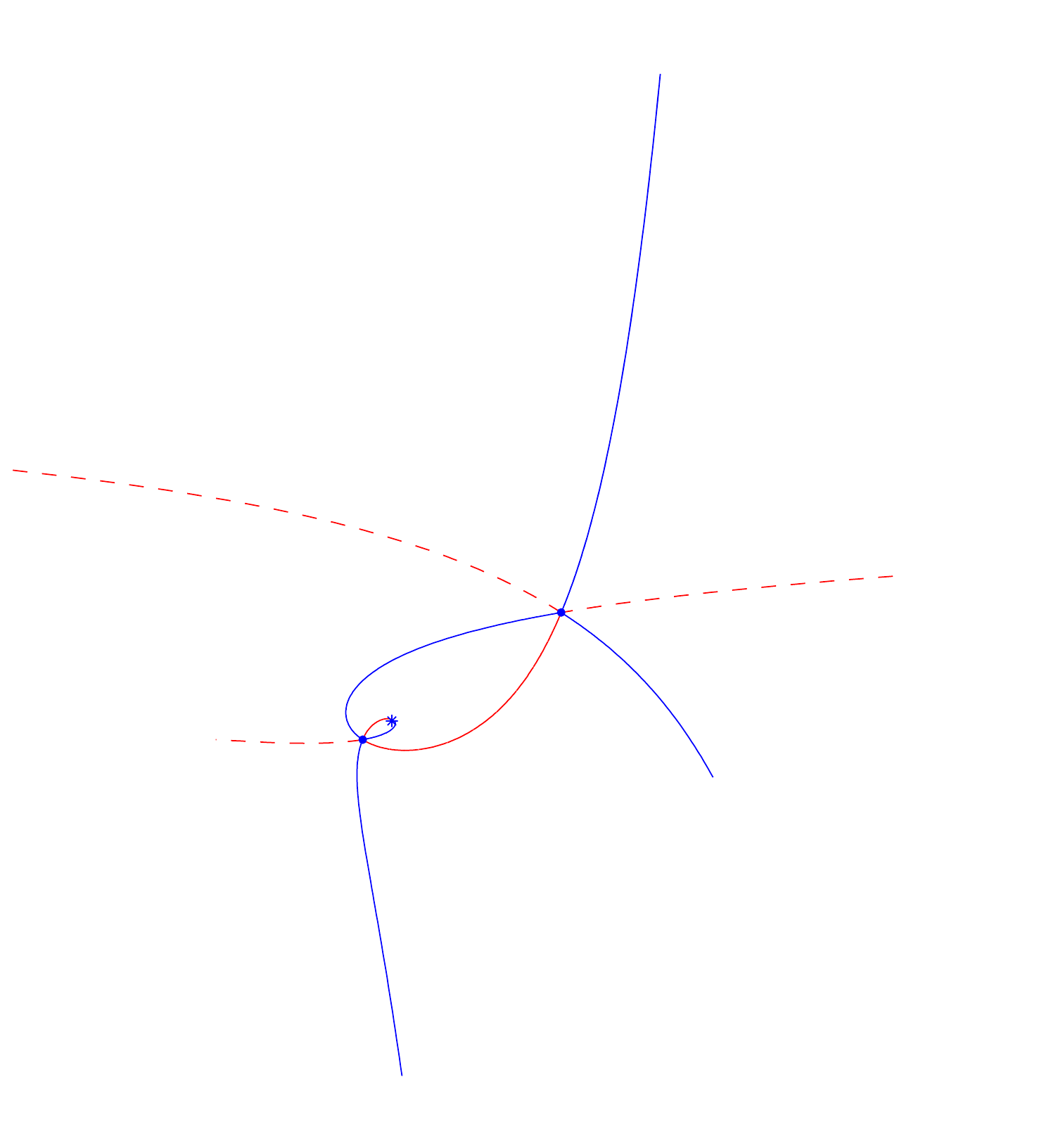}%
    \put(35,37){\scriptsize $0 $}
  \put(26,32){\scriptsize $\zeta_- $}
    \put(48,42){\scriptsize $\zeta_+ $}
     \put(58,80){\scriptsize $\sigma_{\uparrow+}$}
          \put(58,29){\scriptsize $\sigma_{\downarrow+}$}
             \put(27,17){\scriptsize $\sigma_{-}$}
\put(33,46){\scriptsize $\gamma_A$}
\end{overpic}} &
\hspace{-1.cm}
\mbox{\begin{overpic}[scale=0.35]%
{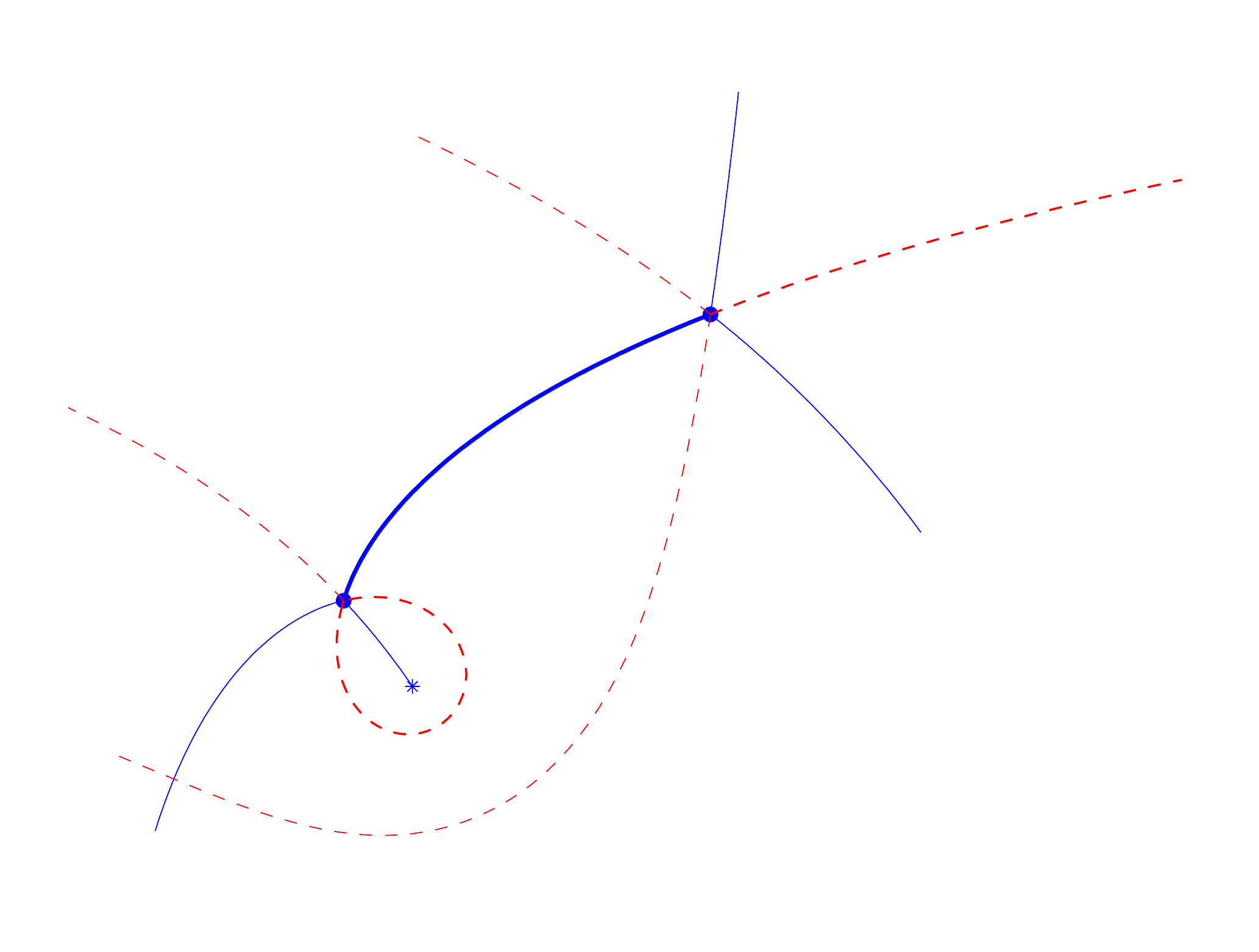}%
    \put(32,17){\scriptsize $0 $}
  \put(23,29){\scriptsize $\zeta_- $}
    \put(59,49){\scriptsize $\zeta_+ $}
 \put(59,65){\scriptsize $\sigma_{\uparrow+}$}
          \put(68,31){\scriptsize $\sigma_{\downarrow+}$}
             \put(12,20){\scriptsize $\sigma_{-}$}
 \put(30,25){\scriptsize $\sigma_{0}$}
\put(33,43){\scriptsize $\gamma_A$}
\end{overpic}}&
\hspace{-1.7cm}
\mbox{\begin{overpic}[scale=0.35]%
{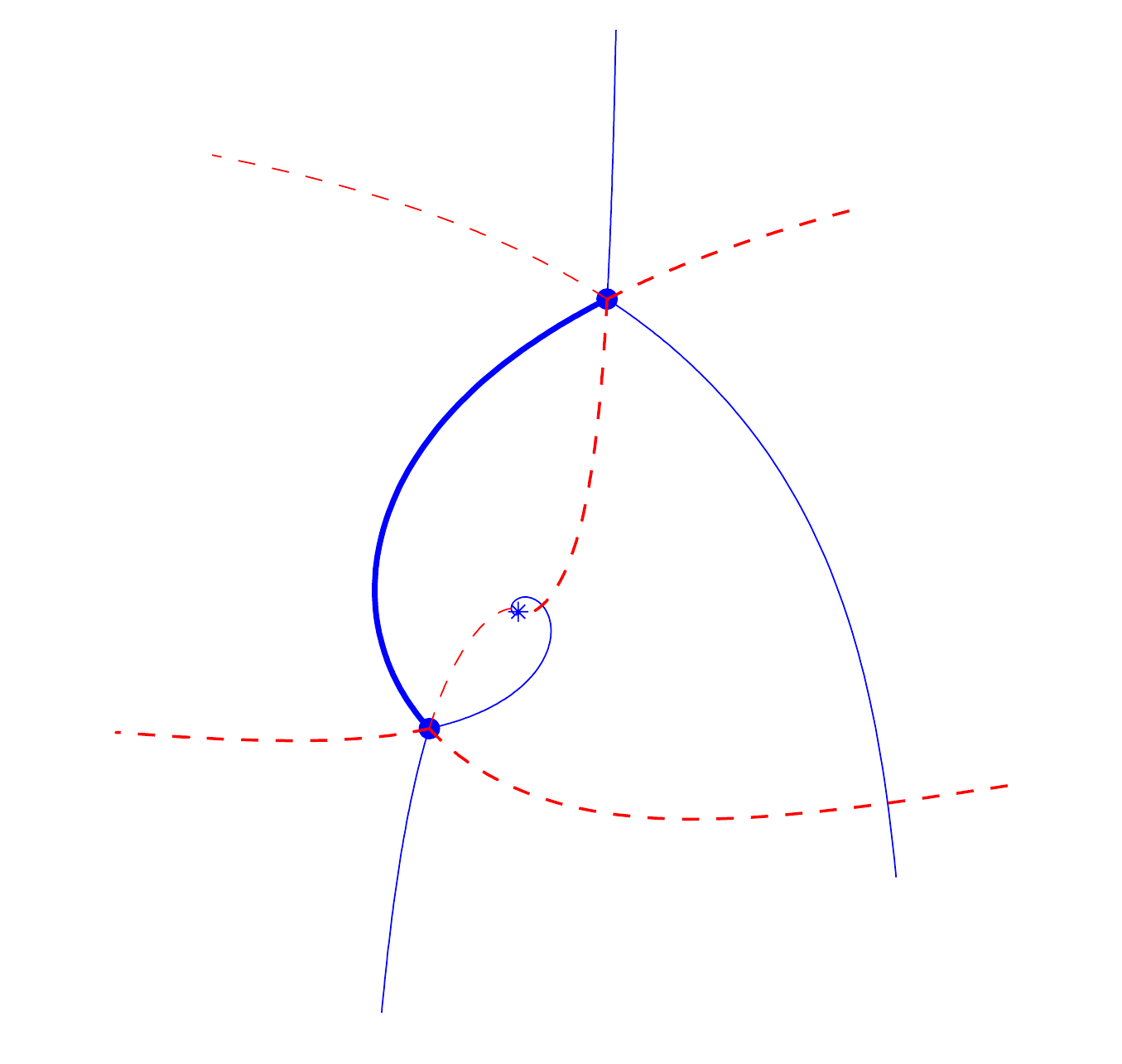}%
   \put(43.5,35.5){\scriptsize $0 $}
  \put(39,25){\scriptsize $\zeta_- $}
    \put(56,69){\scriptsize $\zeta_+ $}
 \put(57,85){\scriptsize $\sigma_{\uparrow+}$}
          \put(78,40){\scriptsize $\sigma_{\downarrow+}$}
             \put(35,10){\scriptsize $\sigma_{-}$}
    \put(50,35){\scriptsize $\sigma_{0}$}
\put(28,53){\scriptsize $\gamma_A$}
\end{overpic}}
\end{tabular}
\caption{Typical structure of the critical graph $\Gamma_A$ for the trajectories (continuous line) and orthogonal trajectories  (dotted line) of $\varpi_A$ for  $\Re (A+1) =0$ (left), $\Re A=0$ (center) and in the rest of the cases with $A\in \C_+$.}
\label{fig:globalstructure}
\end{figure}
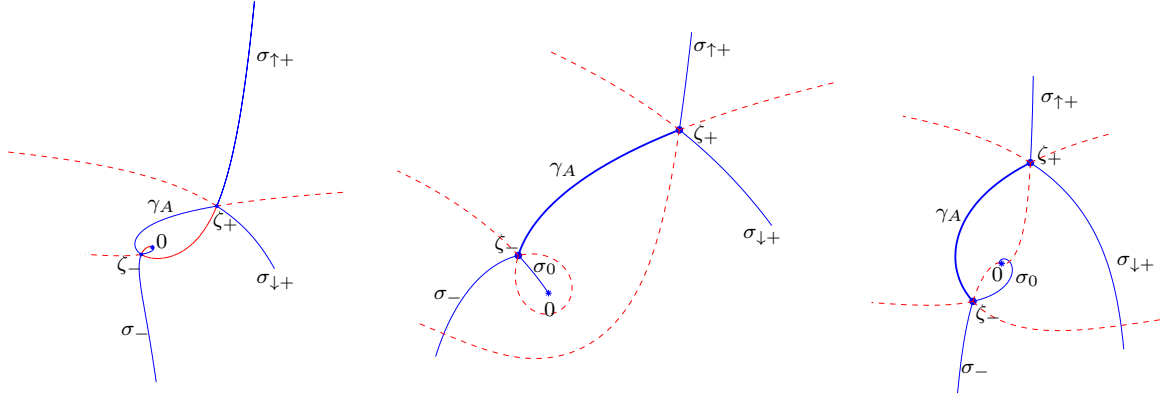

For $A\in \R$ the structure of $\Gamma_A$ has been thoroughly discussed in \cite{Kuijlaars/McLaughlin:01, Kuijlaars/Mclaughlin:04, MR1858305}. When $A>-1$, $\zeta_\pm$ are real,  one of the short trajectories is the real segment joining  $\zeta_-$ with $\zeta_+$, and the second one is a closed loop emanating from the leftmost zero $\zeta_-$ and encircling $0$. When $A< -1$, $\zeta_-$ is the complex conjugate of $\zeta_+$, there are two different short trajectories, joining $\zeta_-$ with $\zeta_+$, such that their union separates the origin from infinity. Case $A=-1$ is degenerate, when $\zeta_-=\zeta_+=1$. Hence, in our analysis we will concentrate on the case $A\notin \R$, and thus with our previous assumption,  $\Im (A)>0$, although the final results include $A\in \R$ as the limit case.

For the benefit of the reader we describe first the general scheme of the proof of Theorem~\ref{thm:1} before giving the technical details. The proof actually spans several lemmas and comprises the following steps: 
\begin{itemize}
\item Since we are interested in the (unique) short critical trajectory connecting both zeros of $D_A$, we start by studying the dependence of $\zeta_\pm$ from the parameter $A$;
\item from the perspective of the existence of this trajectory it is also important to calculate the possible values of the integral $ \int_{\zeta_-}^{\zeta_+}   t^{-1 } \sqrt{D_A(t)}_+ \, dt$; this is done in Lemma~\ref{lemma:integrals};
\item a key fact that allows us to ``test'' the admissibility of a hypothetical structure of $\Gamma_A$ is the so-called Teichm\"uller's lemma (formula~\eqref{teich1}).  Two of it straightforward consequences are Lemmas~\ref{lem:directions} and \ref{lem:2spirals}, which discuss the case of two infinite critical trajectories diverging simultaneously either to $0$ or infinity. Their combination yields the existence of exactly one infinite critical trajectory  (joining a zero of $\varpi_A$ with the origin) and of exactly one short trajectory connecting $\zeta_-$ and $\zeta_+$ (Corollary~\ref{cor:existenceTrajectories});
\item we conclude the proof of the structure of $\Gamma_A$ appealing again to the Teichm\"uller's lemma.
\item the claim made in the statement of Theorem~\ref{thm:1} about the intersection of the short trajectory with the straight segment, joining $\zeta_-$ and $\zeta_+$, requires an additional calculation, see Lemma~\ref{lemma:convexity}.
\item As a bonus, we also discuss an alternative argument for the existence of a short trajectory joining $\zeta_\pm$, which might be applicable to more general situations (Remark~\ref{remark2}).
\end{itemize}
We finish this section discussing the structure of the orthogonal trajectories of $\varpi_A$.

Now we turn to the detailed proofs, clarifying the possible location of the zeros $\zeta_\pm$ on the plane (see Figure~\ref{fig:parabola}):
\begin{lemma}\label{lemma:locationzeros}
Let $\tau$ be the locus of the parabola on $\C$ given parametrically by $\{  1-t^2+2 i t :\, t\in \R \}$. Then
\begin{itemize}
\item  $A \mapsto \zeta_+(A)$ is the conformal mapping of $\C_+$ onto the domain in $\C_+$ bounded by the ray $[1,+\infty)$ and by $\tau \cap \C_+$. In this mapping, the boundary $[-1,+\infty)$  corresponds to $[1,+\infty)$, while $(-\infty,-1)$ corresponds to $\tau \cap \C_+$.
\item $A \mapsto \zeta_-(A)$ is the conformal mapping of $\C_+$ onto the domain in $\C$ bounded by the ray $[0,+\infty)$ and by $\tau \cap \C_-$. In this mapping, the boundary $\R_+$ corresponds to itself,  the interval $[-1,0]$ corresponds to $[0,1]$, while $\R_-$  corresponds to $\tau \cap \C_-$. Moreover, the pre-image of $\R_-$ is the parabola in the upper $A$-half plane, given parametrically by $\{  -t^2+2 i t\in \C:\, t\geq 0 \}$.
\end{itemize}
\end{lemma}
\begin{remark}
Loosely, we can describe the dynamics of $\zeta_\pm(A)$ when $A$ travels the boundary $\R$ of $\C_+$ from $-\infty$ to $+\infty$ as follows:  both $\zeta_-(A)$ and $\zeta_+(A)$ come from infinity moving along $\tau$ in the upper ($\zeta_+$) and in the lower ($\zeta_-$) half plane, respectively, and hit the real line at $1$ simultaneously for $A=-1$. At that moment, $\zeta_+(A)$ starts moving to the right along the upper side of $[1,+\infty)$, while $\zeta_-(A)$ moves to the left, traveling the real line until the origin along its lower side. It reaches the origin for $A=0$; after that it ``climbs'' to the upper side of $\R_+$ and moves monotonically to $+\infty$.
\end{remark}
\begin{proof}
Let $\sqrt{z}$ denote the main branch of the square root in $\C\setminus \R_-$. Then it is easy to see that $(1+\sqrt{z})^2$ is a conformal mapping of the  upper half plane onto the domain bounded by the ray $[1,+\infty)$ and the locus of the parabola $\tau \cap \C_+$ (see the shadowed domain in Figure \ref{fig:parabola}, left). From \eqref{formulaZeros} it follows that this is precisely the domain of $\zeta_+(A)$ when $\Im(A)>0$; furthermore, $\zeta_+(A)$ is real (and thus, $\geq 1$) if and only if $A\geq -1$.

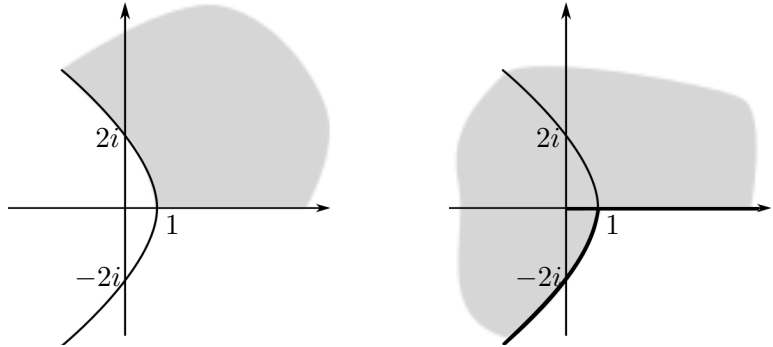
\begin{figure}[htb]
\centering
\begin{tabular}{ll}
\mbox{\begin{overpic}[scale=1]%
{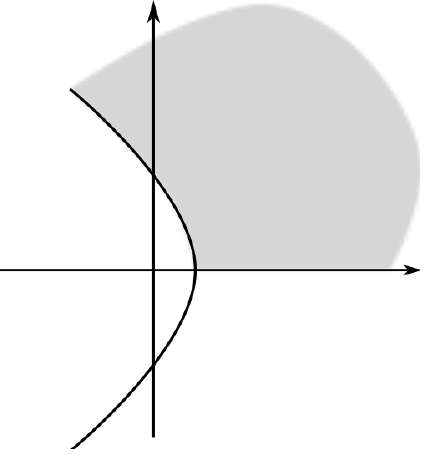}%
       \put(45,33){\small $ 1 $}
       \put(25,58){\small $ 2i $}
       \put(19,18){\small $ -2 i$}
\end{overpic}} &
\hspace{1.0cm}
\mbox{\begin{overpic}[scale=1]%
{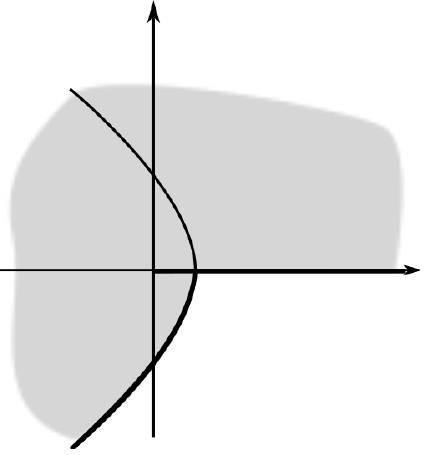}%
 \put(45,33){\small $ 1 $}
       \put(25,58){\small $ 2i $}
       \put(19,18){\small $ -2 i$}
       \end{overpic}}
\end{tabular}
\caption{Domain of $\zeta_+(A)$ (left) and $\zeta_-(A)$ (right) for $\Im(A)>0$. The zero $ \zeta_-(A)$ is in the lower half plane if and only if $ (\Im(A))^2<- 4 \Re(A)$ and $\Im(A)>0$. }
\label{fig:parabola}
\end{figure}

On the other hand,  $(1-\sqrt{z})^2$ is a conformal mapping of the  the upper half plane onto the shadowed domain in Figure \ref{fig:parabola}, right,  bounded by the positive semi-axis and the locus of the parabola $\tau \cap \C_-$. Observe also its boundary behavior: $\R_+$ corresponds to itself,  the interval $[-1,0]$ is mapped onto $[0,1]$, while the negative semi axis corresponds to the locus of the parabola mentioned above. Moreover, the pre-image of the negative semi-axis is precisely $\tau \cap \C_+$.

From \eqref{formulaZeros} it follows that this domain is the image of $\C_+$ by $A\mapsto \zeta_-(A)$. In particular,  $\zeta_-(A)$ is positive only when $A\geq -1$, and then $ \zeta_-(A)\leq \zeta_+(A)$. When $A$ tends to a value on $(-1,0)$, the corresponding $\zeta_-(A)$ approaches the interval $(0,1)$ from the lower half plane, but if $A$ tends to a point on $\R_+$, the zero $\zeta_-(A)$ approaches  $\R_+$ from the upper half plane. We also see that $\zeta_-(A)$ is negative when  $A$ lies on the locus of the parabola given parametrically by $\{ z=-t^2+2 i t\in \C:\, t\leq 0 \}$  (in which case,    $\zeta_+(A)\notin \R$). Furthermore,  $ \zeta_-(A)$ is in the lower half plane if and only if
$$
(\Im(A))^2<- 4 \Re(A) \text{ and } \Im(A)>0.
$$
In particular, if $A$ lies in the open first quadrant, both zeros $\zeta_\pm(A)$ are in the upper half plane.
\end{proof}

We consider the family of Jordan arcs in the punctured plane $\C\setminus \{0\}$, connecting $\zeta_+$ and $\zeta_-$. Each such an arc is oriented from $\zeta_-$ to $\zeta_+$, which induces also its ``$+$'' (left) and ``$-$'' (right) sides.
Let us denote by $\mathcal F_A$ the subfamily of such arcs, homotopic in $\C\setminus \{0\}$ to a Jordan arc in  $\C\setminus \R_+$ (in other words, each $\gamma\in \mathcal F_A$ can be continuously deformed in $\C\setminus \{ 0\}$ to an arc not intersecting the positive real axis).
\begin{lemma}\label{lemma:integrals}
Assume that $A\in \C_+$, and that $\gamma$ is a Jordan arc in the punctured plane $\C\setminus \{0\}$, from $\zeta_-$ to $\zeta_+$. Denote by $\sqrt{D_A(z)}$ the single-valued branch of this function in $\C\setminus \gamma$ determined by the condition
\begin{equation}
\label{asymptoticcondition}
\lim_{z\to \infty} \frac{\sqrt{D_A(z)}}{z}=1,
\end{equation}
and let $\sqrt{D_A(z)}_+$ stand for its boundary values on the $+$ side of $\gamma$.

Then
\begin{equation}
\label{integral1}
 \int_{\gamma}   \frac{\sqrt{D_A(t)}_+}{t}\, dt =\begin{cases}
 2\pi i,  & \text{if } \gamma \in \mathcal F_A, \\
 2\pi i (A+1), & \text{otherwise.}
\end{cases}
\end{equation}
\end{lemma}
\begin{proof}
With $A\in \C_+$, assume that $\gamma\in  \mathcal F_A$, and consider  the following auxiliary function
$$
f(z)=\frac{\sqrt{D_A(z)}}{z-A} =\sqrt{ 1 -\frac{4z}{(z-A)^2}},
$$
holomorphic in $\C\setminus ( \gamma\cup \{A \})$, and such that $f(\infty)=1$. According to our analysis of the location of the zeros $\zeta_\pm$, we can find the value of $f(0)$ continuing it analytically from $+\infty$ along the positive semi-axis.

Clearly, $f(0)\in \{-1, +1\}$. Assumption that $f(0)=-1$ implies that the image of $(0,+\infty)$ by $f$ must cross the imaginary axis, so that there exists a value $x>0$ for which
$$
1 -\frac{4x}{(x-A)^2}\leq 0,
$$
or equivalently, if $x^2-2(A+2 r)x+A^2=0$ for $x>0$ and $r\in [0,1)$. Since the discriminant of this last equation is $4r(A+r)$,  this is  impossible for $A\in \C\setminus[-1,+\infty)$, and we conclude that $f(0)=1$, or in other words, $\sqrt{D_A(0)}=-A$ for $\gamma\in \mathcal F_A$. Obviously, if $\gamma\notin \mathcal F_A$, then $\sqrt{D_A(0)}=A$, so that
\begin{equation} \label{sqrt0}
  \sqrt{D_A(0)}= \begin{cases}
 -A,  & \text{if } \gamma \in \mathcal F_A, \\
 A, & \text{otherwise.}
\end{cases}
\end{equation}

Denote
$$
I=\frac{1}{\pi i}\int_\gamma  \frac{\sqrt{D_A(t)}_+}{t}\, dt,
$$
so that
$$
I =\frac{1}{2\pi i} \oint \frac{\sqrt{D_A(t)}}{t}\, dt= \res_{t=0} \frac{\sqrt{D_A(t)}}{t}+ \res_{t=\infty} \frac{\sqrt{D_A(t)}}{t}=  \sqrt{D_A(0)}+ \res_{t=\infty} \frac{\sqrt{D_A(t)}}{t}.
$$
Since
\begin{equation}
 \label{residue_infty}
\frac{\sqrt{D_A(t)}}{t} = 1 -\frac{A+2}{t}+\mathcal O(t^{-2}), \quad t\to \infty,
\end{equation}
we have
$$
\res_{t=\infty} \frac{\sqrt{D_A(t)}}{t}=A+2.
$$
Now \eqref{integral1} follows from \eqref{sqrt0}.
\end{proof}

Recall that in our analysis we assume that $A\in \C_+$. In this situation, due to the local structure of the trajectories, we cannot have closed loops, and we can assert that the critical graph $\Gamma_A$  of $\varpi_A$ consists of at most $6$ trajectories (finite or not), and all the remaining trajectories diverge in both directions, being their limits either $0$ or $\infty$.

We can get additional information about the structure of $\Gamma_A$ using the Teichm\"uller's lemma, see  \cite[Theorem 14.1]{Strebel:84}. We understand by a \emph{$\varpi_A$-polygon} any domain limited only by trajectories or orthogonal trajectories of $\varpi_A$. If we denote by $z_j$ its corners, by $n_j$ the multiplicity of $z_j$ as a singularity of $\varpi_A$ (taking $n_j=1$ if $z_j\in \{ \zeta_-, \zeta_+\}$, $n_j=0$ if it is a regular point, and $n_j<0$ if  it is a pole), and by $\theta_j$ the corresponding inner angle at $z_j$, then
\begin{equation}
\label{teich1}
\sum_j \beta_j = 2 + \sum_i n_i, \quad \text{where } \beta_j = 1-\theta_j \frac{n_j+2}{2\pi},
\end{equation}
and the summation in the right hand side goes along all zeros of $\varpi_A$ inside the $\varpi_A$-polygon. In consequence, for a $\varpi_A$-polygon $\Omega$ not containing $\zeta_\pm$ inside we have
\begin{equation*}
\sum_j \beta_j = \begin{cases}
2, & \text{if } 0\notin \Omega, \\
0, & \text{if } 0\in \Omega.
\end{cases}
\end{equation*}

Straightforward calculation allows us to list the possible values for $\beta_j$'s at corners $z_j$ of a feasible $\varpi_A$-polygon:
\begin{itemize}
\item if $z_j$ is a regular point we have
$$
\beta_j=\begin{cases}
1/2, & \text{if } \theta_j= \pi/2, \\
-1/2, & \text{if } \theta_j= 3\pi/2.
\end{cases}
$$
\item if $z_j\in \{\zeta_-, \zeta_+ \}$, and the two sides of $\Omega$ confluent at $z_j$  belong to the same family of trajectories, we have
$$
\beta_j=\begin{cases}
0, & \text{if } \theta_j= 2\pi/3, \\
-1, & \text{if } \theta_j= 4\pi/3.
\end{cases}
$$
If on the contrary a horizontal and a vertical trajectories intersect at $z_j$ as sides of $\Omega$, we have
 $$
\beta_j=\begin{cases}
1/2, & \text{if } \theta_j= \pi/3, \\
-1/2, & \text{if } \theta_j= \pi, \\
-3/2, & \text{if } \theta_j= 5\pi/3.
\end{cases}
$$
\item At $z_j=\infty$ we  can only have $\theta_j\in \{0, \pi \}$, with $n_j=-4$, so that
$$
\beta_j=\begin{cases}
1, & \text{if } \theta_j= 0, \\
2, & \text{if } \theta_j= \pi.
\end{cases}
$$
\end{itemize}
Let us point out that the Teichm\"uller's lemma is applicable also to $z_j=0$, in which case we always take $\beta_j=1$. Indeed, if two trajectories diverge simultaneously to $z=0$, there is always an orthogonal trajectory (either also diverging to $z=0$, if $\Re A\neq 0$, or looping around the origin otherwise) intersecting both at the right angles. To each of these two corners, formed in this way, it corresponds the value of $\beta=1/2$, so their sum is $1$. Making the intersection points approach the origin we see that in the limit we can consider $\beta_j=1$ for $z_j=0$.

An immediate consequence of the calculations above is the following
\begin{lemma} \label{lem:directions}
Assume that there exist two  infinite critical trajectories $\gamma_1$, $\gamma_2$, emanating from a zero ($\zeta_-$ or $\zeta_+$) and diverging to infinity. Let $\Omega$ be the infinite domain whose boundary is $\gamma_1\cup \gamma_2$, with the inner angle $\theta=2\pi/3$ at this zero. Then $0\notin \Omega$, and
$\gamma_1$ and $\gamma_2$ diverge to $\infty$ in the opposite directions. 

In particular, all three trajectories (or orthogonal trajectories) emanating from a zero cannot diverge simultaneously to $\infty$.
\end{lemma}
\begin{proof}
Indeed, in this case the left hand side in \eqref{teich1} can take only values 1 or 2, with $2$ corresponding to the angle $\pi$ at infinity. The conclusion follows immediately from  identity \eqref{teich1}.

The last conclusion is also straightforward: if all three trajectories emanating from a zero   diverge simultaneously to $\infty$, they split the complex plain into three disjoint domains, but $0$ cannot belong to either one.
\end{proof}

Another consequence of the Teichm\"uller's lemma is the following conclusion:
\begin{lemma}\label{lem:2spirals}
If $A\notin \R$, there cannot exist two infinite critical trajectories  emanating from $\zeta_\pm$ and diverging to the origin.

In the same vein, if $\Re(A+1)= 0$, there cannot exist two   infinite critical orthogonal trajectories  emanating from $\zeta_\pm$ and diverging to the origin.
\end{lemma}
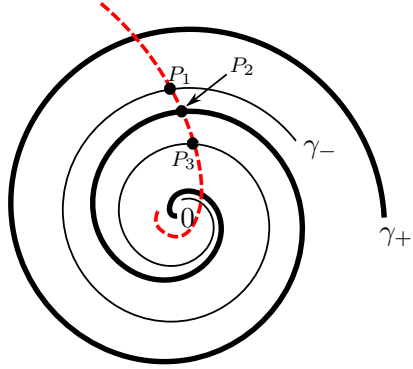
\begin{figure}[htb]
\centering
\begin{overpic}[scale=1.7]%
{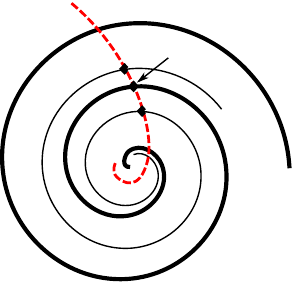}%
 \put(45.5,37.5){\small $0 $}
  \put(78,57){\small $\gamma_- $}
   \put(98,34){\small $\gamma_+ $}
      \put(42,76){\scriptsize $P_1$}
\put(58.5,78.5){\scriptsize $P_2$}
\put(43,53.5){\scriptsize $P_3$}
\end{overpic}
\caption{Local structure of two trajectories (continuous lines)  and orthogonal trajectories (discontinuous line) of $\varpi_A$ near the origin with $ A, iA\notin \R$.}
\label{fig:localstructure3}
\end{figure}

\begin{proof}
As usual,  $A\in \C_+$; consider the case $\Re (A)\neq 0$, and assume that there are two infinite critical trajectories $\gamma_\pm$, both joining a zero with the origin.

Under our assumption on $A$, the local structure of horizontal and vertical trajectories at the origin is the same. Let us denote by $\sigma$ an orthogonal trajectory diverging to the origin. It necessarily intersects both $\gamma_-$ and $\gamma_+$ infinitely many times. We denote by $P_1$ one of the intersections of $\sigma$ with, say, $\gamma_-$. Let $P_2$ be the first time the ray of $\sigma$ emanating from $P_1$ towards $0$ meets $\gamma_+$, and $P_3$ its next intersection with $\gamma_-$ (see Figure \ref{fig:localstructure3}).

Assume first that both $\gamma_\pm$ emanate from the same zero. Consider the bounded $\varpi_A$-polygon limited by the union of the arcs of $\gamma_-$ and $\gamma_+$ joining the zero of $\varpi_A$ with $P_1$ and $P_2$, respectively, and the arc of $\sigma$ joining $P_1$ and $P_2$. Clearly, for such a $\varpi_A$-polygon the right hand side of \eqref{teich1} is equal to $2$. On the other hand, we have seen that the value of $\beta_j$ at the zero can be either $0$ or $-1$, while at $P_1$ and $P_2$, $\beta_j=1/2$. Thus, formula \eqref{teich1} cannot hold.

Suppose now that $\gamma_\pm$ emanate from different zeros.
Let us consider two paths joining $\zeta_-$ and $\zeta_+$. One path is the union of the arc of $\gamma_-$ from $\zeta_-$ to $P_3$, the arc of $\sigma$ from $P_3$ to $P_2$, and the arc of $\gamma_+$ from $P_2$ to $\zeta_+$. The other one is the union of the arc of $\gamma_-$ from $\zeta_-$ to $P_1$, the arc of $\sigma$ from $P_1$ to $P_2$, and the arc of $\gamma_+$ from $P_2$ to $\zeta_+$. It is easy to see that the difference of these two paths consists of the closed curve encircling the origin (the union of the arc of $\gamma_-$ joining $P_1$ and $P_3$ and the arc of $\sigma$ joining $P_1$ and $P_3$ through $P_2$). Hence, $\gamma_-$ and $\gamma_+$ are not homotopic on $\C\setminus \{0\}$. It means that the integral in \eqref{integral1}  along both paths take different values from the right hand side in \eqref{integral1}. In particular, along one of the two paths the integral is purely imaginary, which contradicts the fact that a non-trivial portion of the path goes along the orthogonal trajectory joining $P_2$ with either $P_1$ or $P_3$. This contradiction settles the proof.

All these considerations apply to the case $\Re(A)=0$, with the simplification that now $\gamma_\pm$ are Jordan arcs, and $\sigma$ is a closed Jordan curve, encircling the origin.

Finally, the case $\Re \left( A+1\right) =0$ is analyzed in the same vein, by exchanging the roles
of trajectories and orthogonal trajectories in the last proof.

\end{proof}

Combining Lemmas \ref{lem:directions} and \ref{lem:2spirals} we obtain the following important
\begin{corollary} \label{cor:existenceTrajectories}
If $A\in \C_+$, there exist:
\begin{itemize}
\item exactly one infinite critical trajectory  joining a zero of $\varpi_A$ with the origin; it emanates from $\zeta_-$.
\item exactly one short (finite critical) trajectory connecting $\zeta_-$ and $\zeta_+$.
\end{itemize}
\end{corollary}
\begin{proof}
Consider the three trajectories emanating from a zero, say $\zeta_-$. By Lemma \ref{lem:directions}, they cannot diverge simultaneously to $\infty$, so among them there is at least one short or one infinite critical trajectory diverging to $0$. Notice that a short trajectory that starts and ends at the same zero creates a loop, and thus is boundary of a ring domain (see \cite[Ch.~IV]{Strebel:84})  containing the origin. According to the local structure of trajectories at $z=0$, described above (see Figure \ref{fig:localstructure2}), this is impossible for $A\notin \R$.

 Since the same considerations apply to the other zero of $\varpi_A$, we conclude that among all trajectories emanating from a zero of $\varpi_A$ there is \emph{at most} 4 diverging to infinity, \emph{at most} 1 diverging to the origin, and no loops. This immediately implies the existence of a short  trajectory connecting $\zeta_-$ and $\zeta_+$.

Furthermore, according to Lemma \ref{lemma:integrals}, for $A\notin \R$,
$$
\Re \int_{\zeta_-}^{\zeta_+}   \frac{\sqrt{D_A(t)}}{t}\, dt = 0
$$
can hold only if we integrate along a path in the homotopy class $\mathcal F_A$. Since two different short trajectories cannot belong to the same homotopy class in $\C\setminus \{0\}$, we conclude that there is \emph{at most} one short trajectory. This proves that the short trajectory is exactly one, and implies existence of an infinite critical trajectory connecting a zero with the origin.

Finally, by Lemma \ref{lem:2spirals} and continuous dependence of $\varpi_A$ from $A$ it follows that the infinite critical trajectory diverging to $0$ must emanate from the same zero $\zeta_\pm(A)$ for all $A\in \C_+$. It is sufficient then to analyze the case $  A+i\varepsilon$, with $A>0$ and $ \varepsilon>0$. Recall that for $A>0$, $0<\zeta_-(A)<\zeta_+(A)$; the critical graph $\Gamma_A$  consists of the interval $[\zeta_-(A), \zeta_+(A)]$, two critical trajectories emanating from $ \zeta_+(A)$ and diverging towards $\pm i\infty$, and a closed loop emanating from $\zeta_-(A)$ and enclosing the origin.

We have seen that both $\zeta_-(A+i\varepsilon)$ and $\zeta_+(A+i\varepsilon)$ are in the upper half plane, close to their original positions $\zeta_\pm(A)$. By continuity, for small values of $\varepsilon>0$ there still are two critical trajectories emanating from $ \zeta_+(A+i\varepsilon)$ and diverging towards $\pm i\infty$. However, the closed loop can no longer exist; it breaks into two critical trajectories starting at $ \zeta_-(A+i\varepsilon)$. As we have seen, one of these trajectories must diverge to the origin.
\end{proof}

The combination of Corollary \ref{cor:existenceTrajectories} with the lemmas above yields the existence of the critical trajectories described in Theorem \ref{thm:1}, for which we will use the notation introduced there.

Consider the $\varpi_A$-polygon, bounded by the two trajectories $\sigma_{\uparrow+}$ and $\sigma_{\downarrow+}$, with the inner angle $2\pi/3$ at $\zeta_+$. By Lemma \ref{lem:directions}, this $\varpi_A$-polygon does not contain the origin, and $\sigma_{\uparrow+}$ and $\sigma_{\downarrow+}$ diverge in the opposite vertical directions (which justifies the notation).

Let us consider now a $\varpi_A$-polygon $\Omega$, bounded by   $\sigma_{-} \cup \gamma_A$ and one of the two trajectories $\sigma_{\uparrow+}$, $\sigma_{\downarrow+}$,  with the inner angle $2\pi/3$ at $\zeta_+$ (see Figure~\ref{fig:globalstructure}).
The analysis based on the Teichm\"uller's lemma above shows that $\Omega$ can be only of one of the following two types:
\begin{itemize}
\item  the inner angle of $\Omega$ at $\zeta_-$ is  $2\pi/3$, $0\notin \Omega$, and the inner angle at $\infty$ is $\pi$, or
\item    the inner angle of $\Omega$ at $\zeta_-$ is   $4\pi/3$, $0\in \Omega$, and the inner angle at $\infty$ is $0$.
\end{itemize}
In particular, both  $\varpi_A$-polygons with the inner angle $2\pi/3$ at $\zeta_+$, bounded by $\sigma_{-} \cup \gamma_A\cup \sigma_{\uparrow+}$ and by $\sigma_{-} \cup \gamma_A\cup \sigma_{\downarrow+}$, respectively, must be of different type. This leaves us with the unique configuration for $\Gamma_A$, up to complex conjugation. This configuration is determined by the asymptotic direction of $\sigma_-$ at infinity, which remains invariant for all $A\in \C_+$.
But for $A<-1$ we know that $\sigma_-$ tends to $-i\infty$ (see \cite{Kuijlaars/McLaughlin:01}), which establishes the corresponding assertion of Theorem \ref{thm:1}.

\begin{remark}\label{remark2}
Let us discuss an alternative approach to the proof of the existence of a short trajectory joining $\zeta_\pm$, which is more general and can be applied in other similar situations. We introduce the set $\mathcal A \subset \C$ defined by
$$
\mathcal A=\{A:\, \text{there exists a short trajectory $\gamma_A$ for $\varpi_A$ joining $\zeta_\pm$} \}.
$$
From \cite{Kuijlaars/McLaughlin:01, Kuijlaars/Mclaughlin:04, MR1858305} it follows that $\R\subset \mathcal A$.

We claim that $\mathcal A$ is open in $\C$. Assume that $A\in \mathcal A\setminus \R$; in particular, the integral in \eqref{integral1} taken along $\gamma_A$ in the appropriate direction is equal to $2\pi i$. By continuity of the quadratic differential $\varpi_A$, for every $\varepsilon >0$ there exists $\delta>0$ such that for any $A'\in \C$ satisfying $|A'-A|<\delta$, there exists a trajectory of $\varpi_{A'}$ emanating from $\zeta_-(A')$ and intersecting the $\varepsilon$-neighborhood $\mathcal U_\varepsilon$ of $\zeta_+(A')$; let us denote it by $\gamma_{A'}$. Obviously, the intersection of $\gamma_{A'}$ with $\mathcal U_\varepsilon$ is an arc of a horizontal trajectory of $\varpi_{A'}$ by definition.
If $\gamma_{A'}$ is not critical (i.e., if it does not intersect $\zeta_+(A')$), then by the local structure of trajectories at a simple zero, we may assume that  $\delta>0$ is small enough so that  $\gamma_{A'}$ is intersected by an orthogonal trajectory $\sigma$ emanating from $\zeta_+(A')$. But in this case the path of integration in \eqref{integral1}  that follows the arc of $\gamma_{A'}$ from $\zeta_-(A')$ to the intersection point  and then continues to $\zeta_+(A')$ along $\sigma$, cannot render a purely imaginary integral. This contradiction shows that the whole small neighborhood of $A$ is still in $\mathcal A$.

On the other hand, $\mathcal A$ is closed in $\C$. Indeed, imagine that $A_n\in \mathcal A$ converge to $A\notin \R$, so that $\zeta_\pm (A_n)\to \zeta_\pm(A)$. For each $A_n$, there exists the (unique) short trajectory $\gamma_{A_n}$ joining $\zeta_\pm(A_n)$. It is easy to see that the limit set of the sequence $\{\gamma_{A_n}\}$ (in the Hausdorff metrics) is either another short trajectory connecting $\zeta_\pm(A)$, or a union of two infinite critical trajectories, connecting each $\zeta_\pm(A)$ with the origin. But the last case is forbidden by Lemma \ref{lem:2spirals}, which concludes the proof that $\mathcal A = \C$.
\end{remark}

Let us prove finally the statement about the intersection of $\gamma_A$ with the straight segment, joining $\zeta_-$ and $\zeta_+$, which we in a slight abuse of notation denote by $[\zeta_-, \zeta_+]$.
\begin{lemma}\label{lemma:convexity}
For $A\in \C\setminus [-1, +\infty)$, 
$$
\gamma_A \cap [\zeta_-, \zeta_+]= \{\zeta_-, \zeta_+\}.
$$
\end{lemma}
\begin{proof}
Using the parametrization $t(s) = r s +\zeta_-$, $r = \zeta_+-\zeta_-=4 \sqrt{A+1}$, $s\in [0,1]$, for the segment $[\zeta_-, \zeta_+]$ we obtain that 
$$
 \int_{\zeta_-}^{z} \frac{\sqrt{D_A(t)}}{t}\, dt = \frac{r^2}{i} \int_{0}^{x} \frac{\sqrt{s(1-s)}}{|t(s)|^2} \, \overline{t(s)}\, ds, \quad z=z(x), \quad x\in [0,1],
$$
where we integrate along $[\zeta_-, \zeta_+]$ from $\zeta_-$ to a certain point $z=z(x)$. In particular,
\begin{align*}
\Re \int_{\zeta_-}^{z} \frac{\sqrt{D_A(t)}}{t}\, dt & = \int_{0}^{x} \frac{\sqrt{s(1-s)}}{|t(s)|^2} \, \Im \left(r^2   \overline{t(s)}\right)\, ds \\ 
& = |r|^2  \int_{0}^{x} \frac{\sqrt{s(1-s)}}{|t(s)|^2} \, \left(\Im (r) s + \Im \left( \frac{r^2}{|r|^2}  \overline{ \zeta_-}\right)\right)\, ds.
\end{align*}
In this expression, $\sqrt{s(1-s)}$ preserves sign as long as the segment (path of integration) does not cross $\gamma_A$. 
Thus, between any two consecutive points of intersection of $[\zeta_-,\zeta_+]$ with $\gamma_A$, function
$$
\ell(s)=\Im (r) s + \Im \left( \frac{r^2}{|r|^2}  \overline{ \zeta_-}\right) 
$$ 
must change sign. But $\ell$ is a linear function in $s$, $\ell \not \equiv 0$ for $A\in \C\setminus [-1, +\infty)$, and in consequence, it can change sign at most once on $[0,1]$. This proves the lemma.
\end{proof}

Although the orthogonal trajectories of $\varpi_A$, defined by
$$
\Im \int^{z} \frac{\sqrt{D_A(t)}}{t}\, dt \equiv \const,
$$
will not play a significant role in the asymptotic analysis of the Laguerre polynomials, their structure has an independent interest. Thus, we conclude this section with its brief description:
\begin{proposition}\label{prop:orthogonal}
Assume $A\in \C_+$ such that $\Re (A+1)\neq 0 $ and $\Re(A) \neq 0$. Then
\begin{enumerate}
\item There are exactly two infinite critical orthogonal trajectories, each joining one of the zeros $\zeta_\pm$ with the origin.
\item The other two orthogonal trajectories emanating from $\zeta_-$ (resp., $\zeta_+$) diverge to $\infty$ in the opposite horizontal directions.
\item A critical orthogonal trajectory starting at either zero in the direction to a connected component of  \/ $\C\setminus \Gamma_A$ not containing the origin among its boundary points, stays in this connected component diverging to infinity.
\end{enumerate}
\end{proposition}
Recall that $\Gamma_A$ is the critical graph of $\varpi_A$, described in Theorem \ref{thm:1}.
\begin{proof}
With our assumption $A\in \C_+$, $A \notin  i\R$, due to the local structure of the orthogonal trajectories, we cannot have closed loops, and we can assert that there are at most $6$ critical orthogonal trajectories (finite or not), and all the remaining orthogonal trajectories diverge in both directions, being their limits either $0$ or $\infty$.

Lemma \ref{lem:directions} shows that all three orthogonal trajectories emanating from a zero $\zeta_\pm$ cannot diverge simultaneously to $\infty$,
while by Lemma \ref{lemma:integrals}, and in particular, by formula \eqref{integral1}, a short orthogonal trajectory joining $\zeta_\pm$ is not possible if $\Re(A+1)\neq 0$. Thus, there should exist at least one orthogonal trajectory connecting each zero with the origin. Arguments used in the proof of Lemma \ref{lem:2spirals} imply that  each zero cannot be connected with $0$ by more than one orthogonal trajectory. This establishes that there are exactly one such trajectory for each zero. The remaining critical trajectories must necessarily diverge to $\infty$, and the conclusion about their asymptotic directions follows from the local structure at infinity.

The conclusion in 3 is a consequence again of the Teichm\"uller's lemma: assume that such an orthogonal trajectory intersects the boundary of the connected component of $\C\setminus\Gamma_A$, forming a $\varpi_A$-polygon, not containing the origin, and with the inner angles $\pi/3$ at the zero of $\varpi_A$, and $\pi/2$ at the intersection of the horizontal and vertical trajectories. But this configuration is forbidden by formula \eqref{teich1}.
\end{proof}

Cases $A\in i\R$ or $(A+1)\in i\R$ are somewhat special.
If $A\in i\R$, the local structure of the orthogonal trajectories at the origin (see Figure~\ref{fig:localstructure2}) shows that there is at least one (and hence, only one) closed critical orthogonal trajectory connecting a zero with itself and encircling the origin. The other orthogonal trajectory emanating from the same zero diverges to $\infty$. All three orthogonal trajectories emanating from the other zero also diverge to $\infty$.

If $A\in \C_+$, $\Re(A+1)=0$, then the horizontal and vertical critical graphs of $\varpi_A$ have the same structure, see again Figure~\ref{fig:globalstructure} illustrating the three generic situations.

\section{Auxiliary constructions} \label{sec:auxiliary}

The distinguished pervasive  short trajectory $\gamma_A$ plays an essential role in the asymptotics of the Laguerre polynomials with complex varying coefficients. As we will show below, it carries (again, asymptotically) the zeros of the rescaled polynomials \eqref{sequencex}, see Figure \ref{fig:trajectoriesWithZeros} below.

Thus, for $A\in \C_+$we use for convenience the notation
\begin{equation}
\label{def_RA}
R_A(z)= \sqrt{D_A(z)}, \quad z \in \C\setminus \gamma_A,
\end{equation}
with the branch of the square root fixed by the asymptotic condition \eqref{asymptoticcondition}. We introduce two analytic functions,
\begin{equation} \label{def:phi}	
\phi(z)=\frac{1}{2} \int^{z}_{\zeta_+} \frac{R_A(t)}{t}\, dt  \quad \text{and} \quad \widetilde \phi(z)=\frac{1}{2}\int^{z}_{\zeta_-} \frac{R_A(t)}{t}\, dt.
\end{equation}
We take $\phi$ defined in the simply connected domain $\mathcal D=\C\setminus ( \sigma_- \cup \sigma_0 \cup \gamma_A )$, while $\widetilde \phi$ is defined in $\widetilde{\mathcal D}=\C\setminus (  \sigma_0 \cup \gamma_A \cup \sigma_{\uparrow+})$, see Figure~\ref{fig:critical_graph}. In this way, both functions are single-valued in their respective domains of definition, as well as in the common domain
 $$
\Omega =  \mathcal D \cap \widetilde{\mathcal D}=\C\setminus ( \sigma_- \cup \sigma_0 \cup \gamma_A \cup \sigma_{\uparrow+}),
 $$
 which consists of two simply-connected disjoint components. We take the curve $\sigma_- \cup   \gamma_A \cup \sigma_{\uparrow+}$ oriented from $-i\infty$ to $+i\infty$, and consistently, the left component of $\Omega$, not containing the origin on its boundary, is $\Omega_+$, while the other component is $\Omega_-$. By the structure of $\Gamma_A$, domain $\Omega_+$ contains the asymptotic direction $-\infty$, while $\Omega_-$, the asymptotic direction $+\infty$.

\begin{proposition}\label{prop:confmapping}
With the notations above,
\begin{equation}\label{mappings}
\begin{split}
\widetilde \phi(\Omega_+) & =  \{ z \in \C: \, \Re z <0  \}, \\
 \widetilde \phi(\Omega_-^{(1)}) & = \{ z \in \C: \, 0< \Re z < \pi \Im(A) \}, \\
 \widetilde \phi(\Omega_-^{(2)}) & = \{ z \in \C: \, \Re z > \pi \Im(A) \}.
\end{split}
\end{equation}
In consequence, $ \widetilde \phi$ establishes a bijection between $\widetilde{\mathcal D}$ and the complex plane $\C$ cut along two vertical slits, joining $0$  and $-\pi i (A+1)$ with $+i\infty$, respectively.

Analogously,
\begin{equation*}
\begin{split}
 \phi(\Omega_+) & =  \{ z \in \C: \, \Re z <0   \}, \\
  \phi(\Omega_-^{(1)}) & = \{ z \in \C: \, - \pi \Im(A)< \Re z <0  \}, \\
  \phi(\Omega_-^{(2)}) & = \{ z \in \C: \, \Re z  >0  \}.
\end{split}
\end{equation*}
In particular,
\begin{equation}\label{signsPhi}
\Re \phi (z) < 0 \text{ for } z \in \Omega_+ \cup \Omega_-^{(1)}.
\end{equation}
\end{proposition}
\begin{proof}
To begin with, consider the conformal mapping $w=\widetilde \phi(z)$ of  the domain $\Omega$. Notice that the critical graph $\Gamma_A$ splits the domain $\Omega$ into three simply connected subdomains: one is $\Omega_+$, and the other two are $\Omega_-^{(1)}$ and $\Omega_-^{(2)}$, such that $\Omega_-^{(2)}$ is bounded only by $ \sigma_{\downarrow+}$ and $ \sigma_{\uparrow+}$, see Figure~\ref{fig:critical_graph}.

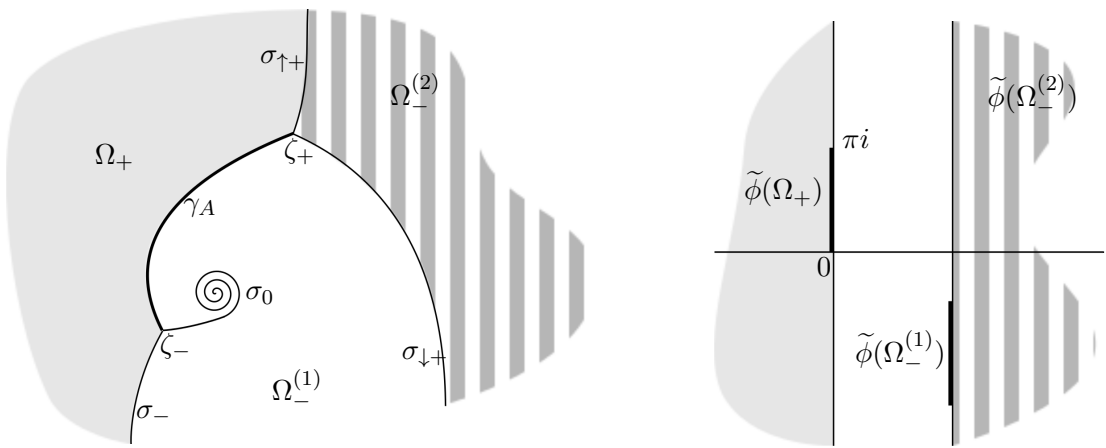
\begin{figure}[htb]
\centering
\begin{tabular}{ll}
\mbox{\begin{overpic}[scale=0.7]%
{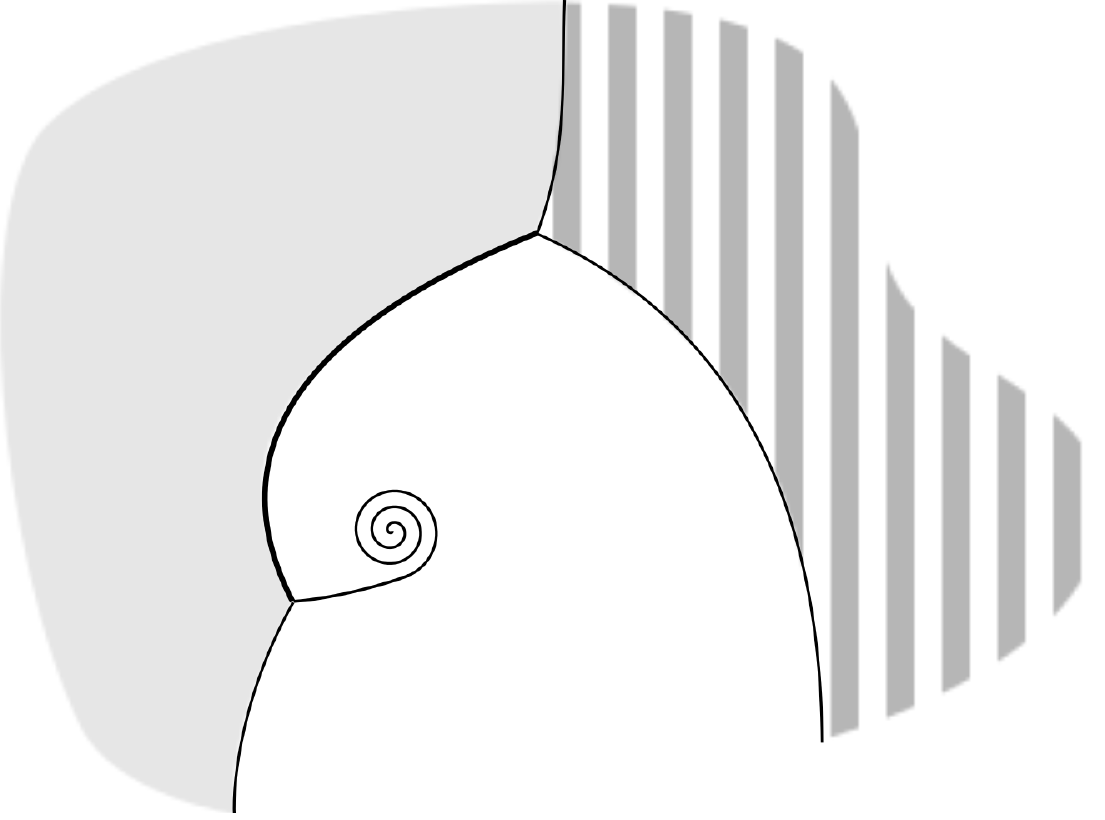}%
       \put(47,49){\small $ \zeta_+ $}
       \put(15,48){\small $ \Omega_+ $}
       \put(45,8){\small $ \Omega_-^{(1)} $}
       \put(65,58){\small $ \Omega_-^{(2)} $}
       \put(26,16){\small $ \zeta_-$}
        \put(30,40){\small $ \gamma_A$}
        \put(22,5){\small $ \sigma_{-}$}
         \put(40.5,25){\small $ \sigma_{0}$}
        \put(67,15){\small $ \sigma_{\downarrow+}$}
        \put(43,65){\small $ \sigma_{\uparrow+}$}
\end{overpic}} & 
\hspace{1.0cm}
\mbox{\begin{overpic}[scale=0.7]%
{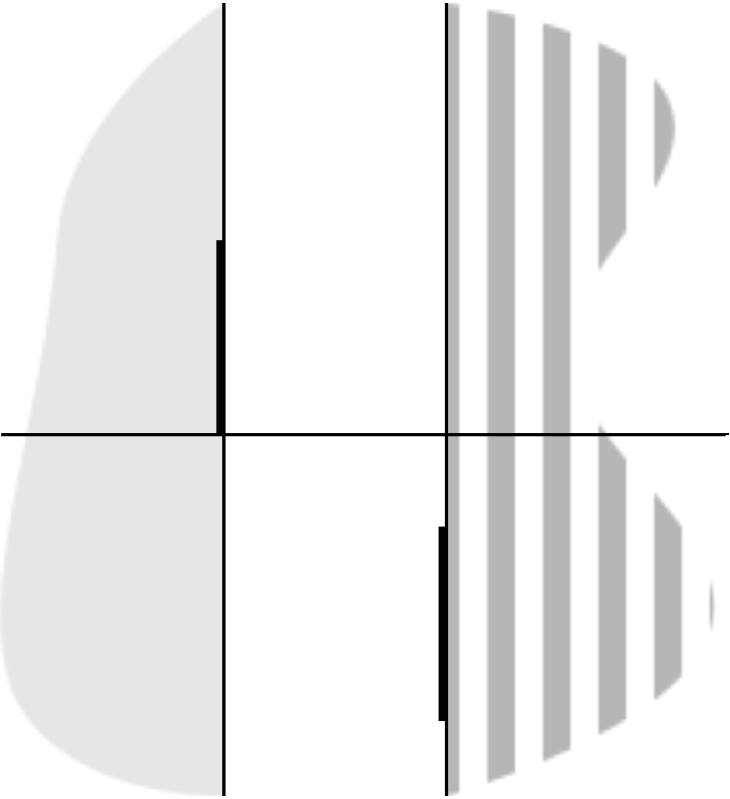}%
       \put(7,58){\small $ \widetilde \phi(\Omega_+) $}
       \put(30,70){\small $ \pi i$}
       \put(24,40){\small $ 0$}
       \put(33,20){\small $ \widetilde \phi(\Omega_-^{(1)}) $}
       \put(64,80){\small $ \widetilde \phi(\Omega_-^{(2)}) $}
       \end{overpic}}
\end{tabular}
\caption{Domains of $\Omega \setminus \Gamma_A$ (left), and their images by the conformal mapping $\widetilde \phi$. }
\label{fig:critical_graph}
\end{figure}

The boundary of $\Omega_+$ consists of critical trajectories only, so that using the definition of $\widetilde \phi$ it is straightforward to see that $\widetilde \phi(\Omega_+)$ is the left half plane. Moreover, by \eqref{integral1},
\begin{equation} \label{limit_from_the_left}	
\lim_{ \Omega_+ \ni  z\to \zeta_+   } \widetilde \phi(z) = \pi i ,
\end{equation}
and the $+$ side of the short trajectory $\gamma_A$ is mapped onto the vertical segment $[0, \pi i]$.

On the other hand, the ``$-$'' side of $\sigma_-$ and the ``$+$'' side of $\sigma_0$ (oriented from the origin to $\zeta_-$) are mapped onto the imaginary axis. There are two asymptotic directions to $\zeta_-$ from $\Omega_-^{(1)}$; for one of them (from the ``$+$'' side of $\sigma_0$) by definition
$$
\lim_{    z\to \zeta_-   } \widetilde \phi(z) = 0.
$$
For the other direction, from the ``$-$'' side of $\sigma_0$,
\begin{equation}
\label{otherside}
\lim_{    z\to \zeta_-   } \widetilde \phi(z) = \frac{1}{2} \oint \frac{R_A(t)}{t}\, dt =\pi i  \res_{t=0} \frac{ R_A(t) }{t} =\pi i R_A(0) = -\pi i A,
\end{equation}
where the contour of integration encircles the origin in the anti-clockwise direction, and where we have used \eqref{sqrt0}. Observe that
$$
\Re(-\pi i A) = \pi \Im(A) >0
$$
for $A\in \C_+$. Thus, the boundary $(\sigma_0)_- \cup (\gamma_A)_- \cup \sigma_{\downarrow +}$ is mapped onto the vertical line in the right half plane, passing through $\pi \Im(A)$. By  \eqref{integral1},
\begin{equation}
 \label{limit_from_the_right}
\lim_{ \Omega_-^{(1)} \ni  z\to \zeta_+   } \widetilde \phi(z) = -\pi i (A+1),
\end{equation}
so that $(\gamma_A)_-$ corresponds by $\widetilde \phi(z)$ to the vertical segment $[-\pi i (A+1), -\pi i A]$. This concludes the proof of \eqref{mappings}.

The corresponding results for function $\phi$ are established using the following connection formula:
\begin{equation*} 
    \phi(z) =  \left\{ \begin{array}{ll}
        \widetilde{\phi}(z) - \pi i, & \quad \text{for } z \in \Omega_+, \\[10pt]
        \widetilde{\phi}(z) + \pi i(1+A), & \quad \text{for } z \in \Omega_-,
    \end{array} \right.
\end{equation*}
which is a direct consequence of formulas \eqref{limit_from_the_left}, \eqref{limit_from_the_right} and the definition of $\phi$:
$$
\lim_{ \Omega_+ \ni  z\to \zeta_+   } \widetilde \phi(z) = \pi i , \quad \lim_{ \Omega_- \ni  z\to \zeta_+   } \widetilde \phi(z) = -\pi i (A+1), \quad  \lim_{ \Omega_\pm \ni  z\to \zeta_+   }  \phi(z) = 0.
$$

\end{proof}

Observe also that with Proposition \ref{prop:confmapping} we conclude  the proof of Theorem \ref{thm:1}. Our main result of the first part of this section is the following
\begin{corollary}\label{prop:progressive}
There exist two analytic curves, $\Sigma_+$ and $\Sigma_-$, in $\Omega_-\setminus \R_+$, emanating from $\zeta_+$ and $\zeta_-$, respectively, and diverging to infinity in the asymptotic direction $+\infty$, and such that
$$
\Re \phi(z)>0 \text{ for } z \in \Sigma_-, \quad \text{and}\quad \Re \widetilde \phi(z)>0 \text{ for } z \in \Sigma_+.
$$
\end{corollary}
From the considerations at the end of Section~\ref{sec:trajectories} it follows that one instance of $\Sigma_+$ could be the orthogonal trajectory of $\varpi_A$ emanating from $\zeta_+$ and diverging in the $+\infty$ direction, see Figure~\ref{fig:globalstructure} or Figure~\ref{fig:trajectoriesWithZeros} below. However, this fact is not so relevant; what really matters for our further analysis is the existence of both curves stated in this Corollary.

In what follows we will make use of the contour
\begin{equation} \label{def:SigmaA}	
\Sigma_A = \Sigma_- \cup \gamma_A \cup \Sigma_+,
\end{equation}
oriented clockwise, in such a way that $\Sigma_A$ is not homotopic to a point in $\C\setminus \R_+$, and the origin remains on the right of the curve. Observe that $\Sigma_A$ is not uniquely determined due to the freedom in the choice of $\Sigma_\pm$, although $\gamma_A$ is. 

\medskip

We turn now to the second goal of this section. The main tools for the study of the weak asymptotic behavior of
polynomials satisfying a non-hermitian orthogonality were
developed in the seminal works of Stahl \cite{Stahl:86} and Gonchar
and Rakhmanov \cite{Gonchar:87}. They showed that when the complex
analytic weight function depends on the degree of the polynomial, the
limit zero distribution is characterized by an equilibrium problem on
a compact set in the presence of an external field; this compact set
must satisfy a symmetry property with respect to the external field.

For any positive Borel measure $\mu$ on $\C$, such that $$ \int _{| z
| \geq 1} \log | z | \, d\mu(z) < +\infty\,, $$ we can define its
logarithmic potential
$$
V^\mu(z)=-\int \log|t-z| \, d\mu(t)\,.
$$
Given $A$, $\Im(A)>0$, consider  the harmonic \emph{external field}
$$ \psi(z)= -  \frac{\Re(A)}{2} \log |z|+ \frac{\Im(A)}{2}\arg(z)+\frac{\Re(z)}{2}
  $$
in $\C\setminus \R_+$, where we take the main branches of $\log$ and $\arg$.  Let also the contour $\Sigma_A$ be as defined in \eqref{def:SigmaA}.

It is known that the (unique) probability \emph{equilibrium measure} $\mu$
on  $\Sigma_A$  in the external field $\psi$ can be characterized by the property
\begin{equation}\label{equilibrium}
V^\mu(z)+\psi(z)
  \begin{cases}
    = \ell =\const, & \text{for $z \in \supp(\mu)$}, \\
    \geq \ell & \text{for $z \in \Sigma_A$},
  \end{cases}
\end{equation}
 where $\ell $ is the  equilibrium constant; for details see
e.g.\ \cite{Gonchar:87} or \cite{Saff:97}.

Both the measure $\mu$ and its support $\supp(\mu)$ have the \emph{$S$-property}
in the external field $\psi$  if at any interior point $\zeta$ of $\supp(\mu)$,
\begin{equation}\label{simetria}
  \frac{\partial (V^\mu+\psi)}{\partial n_-}\,(\zeta) =
  \frac{\partial (V^\mu+\psi)}{\partial n_+}\,(\zeta)\,,
\end{equation}
where $n_-=-n_+$ are the normals to $\supp(\mu)$.

\begin{proposition}\label{prop:eq_measure}
The absolutely continuous measure
\begin{equation}
\label{eq:defMu}
d\mu(z) = d\mu_A(z)=\frac{R_A(z)_+}{2 \pi i z}\,dz, \quad z \in \gamma_A
\end{equation}
is the equilibrium measure on $\Sigma_A$ in the external field $\psi$.
\end{proposition}
Function $R_A$ was defined in \eqref{def_RA}. Recall that as usual, $\gamma_A$ is oriented from $\zeta_-$ to $\zeta_+$, and $R_A(z)_+$ is the boundary value on its left side.
\begin{proof}
By the definition of $\gamma_A$, the right hand side in \eqref{eq:defMu} is real-valued and non-vanishing along $\gamma_A$. Taking into account \eqref{integral1} we conclude that $\mu$ is a probability measure on $\gamma_A$. Furthermore, as in the proof of Lemma~\ref{lemma:integrals} and using \eqref{residue_infty}, for $z\notin \gamma_A$,
\begin{align*}
\int \frac{d\mu(t)}{t-z} & =\frac{1}{4\pi i} \oint_{\gamma_A} \frac{R_A(t)}{t(t-z)}\, dt= \frac{1}{2} \left( \res_{t=0} \frac{R_A(t)}{t (t-z)}+ \res_{t=\infty} \frac{R_A(t)}{t (t-z)}+ \res_{t=z} \frac{R_A(t)}{t (t-z)} \right) \\
&  =  \frac{1}{2} \left(  \frac{A}{ z}-1+ \frac{R_A(z)}{z} \right) .
\end{align*}
Thus,
\begin{equation*} 
V^{\mu}(z) = \const + \frac{1}{2} \Re \int_{\zeta_+}^z \left(  \frac{A}{ t}-1+ \frac{R_A(t)}{t} \right)dt, \quad z \in \C\setminus \gamma_A.
\end{equation*}
Thus, for $z \in \C\setminus (\gamma_A\cup \sigma_- \cup \sigma_0\cup \R_+)$,
\begin{equation} \label{identityV}	
V^{\mu}(z) + \psi(z)=V^{\mu}(z) + \frac{1}{2} \Re \left( - A \log(z) + z  \right) =  \frac{1}{2} \left(-\ell + \Re \phi(z)\right),
\end{equation}
where $\phi$ was defined in \eqref{def:phi}, and $\ell$ is an appropriately chosen real constant. It remains to use that the last term in the right hand side vanishes on $\gamma_A$ and is strictly positive on $\Sigma_\pm$ (see Corollary~\ref{prop:progressive}) to conclude that for $\mu$ in \eqref{eq:defMu} characterization \eqref{equilibrium} holds.

Finally, since $V^{\mu}(z) + \psi(z) - \ell$, harmonic in $\C\setminus (\gamma_A\cup \R_+)$, is identically $0$ on $\gamma$ and is the real part of the analytic function
$$
\mathcal W(z)=\frac{1}{2}  \int_{\zeta_-}^z   \frac{R_A(t)}{t}  dt,
$$
satisfying $\mathcal W'_-(z)=-\mathcal W'_+(z)$ on $\gamma_A$, the symmetry property \eqref{simetria} easily follows from the Cauchy--Riemann equations.
\end{proof}

In the next section a relevant role will be played by the so-called $g$-function, i.e.~the complex potential of the equilibrium measure $\mu_A $ on $\gamma_A$:
\begin{equation} \label{gfunction}
    g(z)=g(z,A) = \int_{\Gamma} \log(z-s) \, d\mu(s),
      \qquad z \in \mathbb C \setminus (\gamma_A \cup \Sigma_-),
\end{equation}
where for each $s$ we view $\log (z-s)$ as an analytic function of the variable
$z$, with branch cut emanating from $z=s$; the cut is taken
along $\gamma_A \cup \Sigma_-$.

\section{Asymptotics of Laguerre polynomials} \label{sec:RHanalysis}

Now we have all ingredients to formulate and prove the asymptotic results for  the rescaled generalized Laguerre polynomials $p_n$, defined in \eqref{sequencex} in Section \ref{sec:intro}, under the assumption \eqref{limits}. The analysis is based on  the non-hermitian orthogonality conditions satisfied by these polynomials (Section \ref{sec:orthogonality}). For the weak asymptotics (or limiting zero distribution) we can use an analogue of the Gonchar--Rakhmanov--Stahl's results, while the strong asymptotics is derived from  the corresponding Riemann-Hilbert characterization of these polynomials (Section \ref{sec:orthogonality}) using the non-linear steepest descent method of Deift and Zhou.

\subsection{Orthogonality conditions and Riemann--Hilbert characterization} \label{sec:orthogonality}

Throughout this section we assume that $\alpha \in \C \setminus \R$.
Then,  the generalized Laguerre polynomials
$L_n^{(\alpha)}$ satisfy a non-hermitian orthogonality in the complex plane, see \cite{MR1858305}. Namely,
\begin{theorem}\label{ThOrt}
Let $\Sigma$ in $\C \setminus [0,+\infty)$ be an unbounded Jordan arc diverging in both directions toward  $+\infty $, $n \in \N$, and
$\alpha \in \C \setminus \R$.
Then
\begin{equation} \label{eq21}
    \int_{\Sigma} z^k L_n^{(\alpha)}(z) z^{\alpha} e^{-z} dz = 0,
    \quad \mbox{ for } k = 0, 1, \ldots, n-1.
\end{equation}
In addition, 
\begin{equation} \label{eq22}
    \int_{\Sigma} z^n L_n^{(\alpha)}(z) z^{\alpha} e^{-z} dz \neq 0.
\end{equation}
\end{theorem}
Although the selection of the branch of $z^{\alpha}$ is not relevant, for the sake of definiteness we take here  $z^{\alpha} = |z|^{\alpha} \exp(i \alpha \arg z)$, with $\arg z \in [0, 2 \pi)$.

\begin{proof} In the proof we use $f^{(k)}$ to denote the $k$-th derivative of $f$
and $w(z;\alpha):=z^{\alpha} e^{-z}$.

By the Rodrigues formula (\ref{RodrLag}),
\begin{equation}\label{RodrLag2}
L_n^{(\alpha)}(z)=\frac{(-1)^n}{n!} \, \frac{w^{(n)}(z;
n + \alpha)}{ w(z;\alpha)}.
\end{equation}
Integrating in the left hand side of \eqref{eq21} $n$ times by parts and using
(\ref{RodrLag2}), we get
 \begin{equation}\label{partes}
\begin{split}
  \int_{\Sigma} z^k L_n^{(\alpha)}(z) z^{\alpha} e^{-z} dz= \frac{(-1)^n}{n!} \,
 \sum_{j=0}^{n-1} (-1)^j \big[z^k \big]^{(j)}    w^{(n-j-1)}(z; n+\alpha)
 \bigg|_{\Sigma} \\
 +\frac{1}{n!} \, \int_{\Sigma} \big[z^k \big]^{(n)} w(z; n+\alpha)\,
 dz\,.
\end{split}
\end{equation}
Since $w(z; \alpha)$ is single-valued on $\Sigma$,
\begin{equation*}
\big[z^k \big]^{(j)}    w^{(n-j-1)}(z; n+\alpha)
\bigg|_{\Sigma}=0\,, \qquad \mbox{for } 0 \leq j \leq n-1\,.
\end{equation*}
Thus, if $k\leq n-1$, all the terms in the right-hand side of
(\ref{partes}) vanish, and (\ref{eq21}) follows.

Furthermore, for $k=n$, we also get
\[ \int_{\Sigma} z^n L_n^{(\alpha)}(z) z^{\alpha} e^{-z} \, dz =
    \frac{(-1)^n}{n!} \int_{\Sigma} z^n w^{(n)}(z; n+\alpha)
    dz = (-1)^n \int_{\Sigma} z^{\alpha+n} e^{-z} \, dz.
\]
We deform $\Sigma$ to the positive real axis to obtain
\begin{eqnarray} \nonumber
\int_{\Sigma} z^n L_n^{(\alpha)}(z)  z^{\alpha} e^{-z} \, dz
& = & (-1)^n \left(1- e^{2\pi i \alpha}\right) \int_0^{\infty} z^{\alpha + n} e^{-z} \, dz \\
\label{eq23}
&= &  (-1)^{n+1} 2i e^{\pi i \alpha} \sin (\pi \alpha)  \Gamma(\alpha + n + 1),
\end{eqnarray}
where $\Gamma$ denotes the Gamma function.
By analytic continuation the integral in (\ref{eq22}) is equal to (\ref{eq23})
for every $\alpha$, and (\ref{eq22}) follows.
\end{proof}

\begin{figure}
\centerline{\includegraphics[width=25cm]{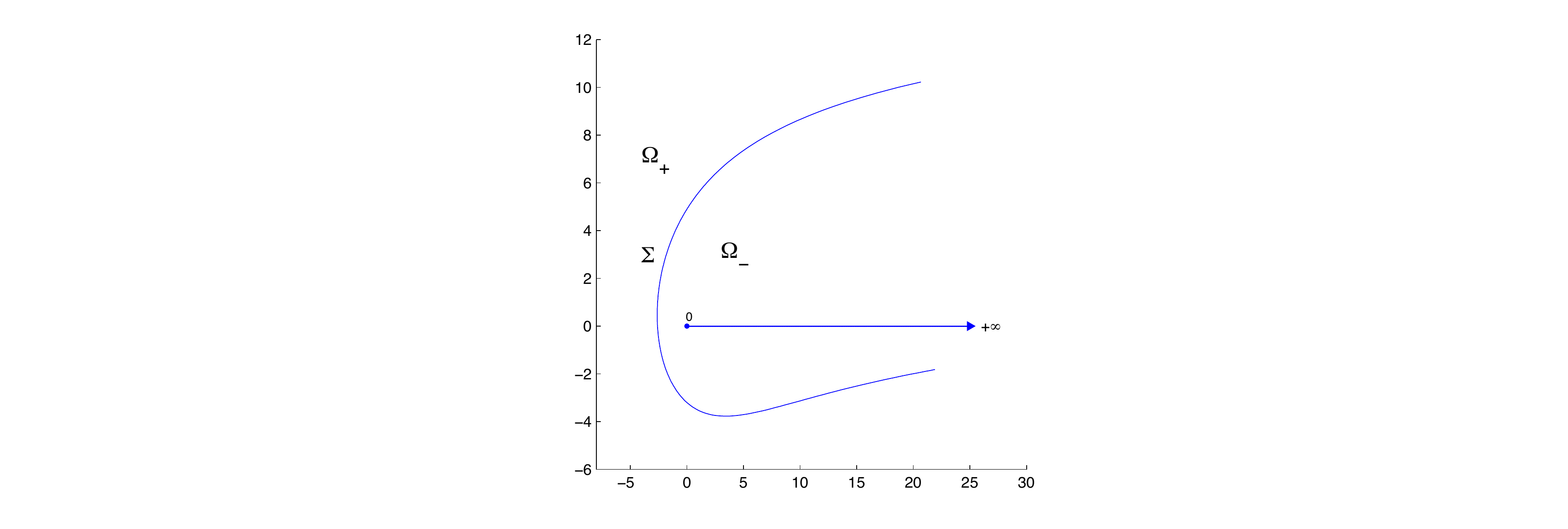}}
\caption{A contour $\Sigma $.}
\label{fig:CurvaOrtog}
\end{figure}

Consider now the monic rescaled polynomials
\begin{equation} \label{eq24}
    P_n(z) = \frac{(-1)^n n!}{n^n} L_n^{(\alpha)}(nz) = \frac{(-1)^n n!}{n^n} \, p_n(z), \qquad n = 0, 1, \dots
\end{equation}
By Theorem \ref{ThOrt},  $ P_n(z)$ verifies the orthogonality
\begin{equation*} 
    \int_{\Sigma} z^k P_n(z) z^{\alpha} e^{-nz}\, dz
    \left\{ \begin{array}{ll}
    = 0, & \quad  \mbox{ for } k = 0, 1, \ldots, n-1, \\[10pt]
    \neq 0, & \quad \mbox{ for } k = n,
    \end{array} \right.
\end{equation*}
for a contour $\Sigma$ specified in Theorem \ref{ThOrt} (see Figure~\ref{fig:CurvaOrtog}). By the classical work of Fokas, Its, and Kitaev \cite{Fokas92}, this yields a characterization of the polynomial $P_n$ in terms of a Riemann--Hilbert problem:
determine a $2 \times 2$
matrix valued function $Y : \mathbb C \setminus \Sigma \to \mathbb C^{2 \times 2}$ satisfying the following conditions:
\begin{enumerate}
\item[(a)] $Y(z)$ is analytic for $z \in \mathbb C \setminus \Sigma$,
\item[(b)] $Y(z)$ possesses continuous boundary values for $z \in \Sigma$. If $\Sigma$ is oriented clockwise, and  $Y_+(z)$, $Y_-(z)$ denote the non-tangential boundary values of $Y(z)$ on the left and right sides of $\Sigma$, respectively, then
\begin{equation*} 
    Y_+(z) = Y_-(z) \begin{pmatrix} 1 & z^{\alpha} e^{-nz} \\ 0 & 1
    \end{pmatrix},
    \qquad \mbox{ for } z \in \Sigma.
\end{equation*}
\item[(c)] $Y(z)$ has the following behavior as $z \to \infty$:
\begin{equation*} 
    Y(z) = \left( I + O\left(\frac{1}{z}\right) \right)
    \begin{pmatrix}  z^n & 0 \\ 0 & z^{-n}
    \end{pmatrix},
    \qquad \mbox{ as } z \to \infty, \ z \in \mathbb C \setminus \Sigma.
\end{equation*}
\end{enumerate}
The unique solution of this  problem for $Y$ is given by (see \cite{Fokas92})
\begin{equation*} 
    Y(z) = \begin{pmatrix}
    P_n(z) & \displaystyle{\frac{ 1}{2\pi i}  \int _{\Sigma}}
        \frac{P_n (\zeta) \zeta^{\alpha}e^{-n\zeta}}{\zeta - z} \, d\zeta \\[10pt]
    {  d_{n-1}} P_{n-1}(z) & \displaystyle{\frac{d_{n-1}}{2\pi i}  \int_{\Sigma}}
     \frac{P_{n-1}(\zeta) \zeta^{\alpha} e^{-n\zeta}}{\zeta -z} \, d\zeta
     \end{pmatrix},
    \end{equation*}
where $P_n(z)$ is the monic generalized Laguerre polynomial {\rm (\ref{eq24})}
and  the constant $d_{n-1}$ is chosen to guarantee that for the $(2,2)$ entry of $Y$,
$$
\lim_{z\to \infty} z^n Y_{22}(z) = 1.
$$
The orientation of $\Sigma$ is consistent also with the one used in Section \ref{sec:trajectories}.

\subsection{Asymptotics of varying generalized Laguerre polynomials}  \label{sec:results}

With each $p_n$ in \eqref{sequencex} we can associate its normalized zero-counting measure
$\nu_n=\nu(p_n)$, such that for any compact set $K$ in $\C$, $$
\int_K d\nu_n=\frac{\text{number of zeros of $p_n$ in $K$}}{n} $$
(the zeros are considered taking account of their multiplicity). Weak asymptotics for $p_n$'s studies convergence of the sequence $\nu_n$ in the
weak-$^*$ topology, that we denote by $\stackrel{*}{\longrightarrow}$.

\begin{theorem} \label{teoLag1}
Let the sequence of (generalized) Laguerre polynomials $p_n$ in
(\ref{sequencex}) satisfy (\ref{limits}) with $A \in \C \setminus
\R$. Then
\begin{equation*} 
\nu_n \stackrel{*}{\longrightarrow} \mu\,, \quad n \to \infty,
\end{equation*}
where $\mu$ is the probability measure given in Proposition~\ref{prop:eq_measure}. Its support is a simple analytic arc, namely the short trajectory $\gamma_A$ described in Theorem~\ref{thm:1}.
\end{theorem}

Regarding the strong asymptotics we have a result very similar to that in \cite{Kuijlaars/McLaughlin:01}, but referring now to the short trajectory $\gamma_A$. Because of the $n$-dependence of $\alpha_n$, all of the notions and results introduced
previously are $n$-dependent. For example, we have that
the curves $\gamma_{A_n}$  are all varying with $n$,
and so we denote them by $\gamma_n$; clearly, they tend to the
limiting curve $\gamma_A$.
Likewise, we have that the functions $R_{A_n}$, $g$, $\phi$, and $\tilde{\phi}$,
as well as all matrix-valued functions are $n$-dependent, and we also use a
subscript $n$ to denote their dependence on $n$, using the notation without
subscript $n$ when referring to the limiting case.

\begin{theorem} \label{thm:strongA}
For the rescaled generalized
Laguerre polynomials $p_n$ as $n \to \infty$,
\begin{enumerate}
\item[\rm (a)] {\bf (Asymptotics  away from $\gamma_A$.)} \\
Uniformly for $z$ in compact subsets of $\mathbb C \setminus \gamma_A$,
we have as $n \to \infty$,
\begin{equation*} 
    p_n(z) = \frac{(-n)^n}{n!} e^{ng_n(z)}
        \left( \frac{1 + R_n'(z)}{2} \right)^{1/2} \left(1 + \mathcal O \left(\frac{1}{n}\right) \right).
\end{equation*}
\item[\rm (b)] {\bf (Asymptotics on $+$-side of $\gamma_n$, away
from endpoints.)} \\
Uniformly for $z$ on the $+$-side of $\gamma_n$ away from $\zeta_\pm$,
we have as $n \to \infty$,
\begin{align*} 
   p_n(z)     = & \frac{(-n)^n}{n!} e^{ng_n(z)}
       \left( \frac{1 + R_n'(z)}{2} \right)^{1/2} \\
         & \times
       \left[ 1 - \left(\frac{1-R_n'(z)}{1+R_n'(z)} \right)^{1/2} e^{2n\phi_n(z)}
        + \mathcal O \left(\frac{1}{n}\right) \right].
\end{align*}
\item[\rm (c)] {\bf (Asymptotics on $-$-side of $\gamma_n$, away from endpoints.)} \\
Uniformly for $z$ on the $-$-side of $\gamma_n$ away from $\zeta_\pm$,
we have as $n \to \infty$,
\begin{align*} 
    p_n(z)    = & \frac{(-n)^n}{n!} e^{ng_n(z)}
       \left( \frac{1 + R_n'(z)}{2} \right)^{1/2} \\
         & \times
       \left[ 1 + \left(\frac{1-R_n'(z)}{1+R_n'(z)} \right)^{1/2} e^{2n\phi_n(z)}
        + \mathcal O \left(\frac{1}{n}\right) \right].
\end{align*}
\item[\rm (d)] {\bf (Asymptotics near $\zeta_+$.)} \\
Uniformly for $z$ in a (small) neighborhood of $\zeta_+$, we have as $n \to \infty$,
\begin{align*} 
\lefteqn{p_n(z)   = \frac{(-n)^n}{n!}
    \exp\left( \frac{n}{2} (- A_n \log z + z  + \ell)\right) \sqrt{\pi} } \\
      &   \times \,
    \left[ \left(\frac{z-\bar\beta_n}{z-\beta_n}\right)^{1/4} \left( n^{2/3} f_n(z) \right)^{1/4}
        \Ai(n^{2/3} f_n(z)) \left(1 +\mathcal O\left(\frac{1}{n}\right) \right) \right. \\
      & \qquad 
    \left. - \left(\frac{z-\beta_n}{z-\bar\beta_n} \right)^{1/4}
        \left( n^{2/3} f_n(z) \right)^{-1/4}
        \Ai'(n^{2/3} f_n(z)) \left(1 +\mathcal O\left(\frac{1}{n}\right) \right) \right],
\end{align*}
where
\begin{equation*} 
    f_n(z) =f(z,A_n), \quad f(z,A) = \left[\frac{3}{2} \phi(z) \right]^{2/3}.
\end{equation*}
\end{enumerate}
\end{theorem}
Recall that $R$ was introduced in \eqref{def_RA}, while $g$ was defined in \eqref{gfunction}. See Section~\ref{sec:proofs} below for a more detailed explanation about the conformal mapping $f$. 
\begin{remark}
We can write a similar asymptotic formula at the other endpoint of $\gamma_A$; it will be in terms of the function $\widetilde \phi$ appropriately redefined in a neighborhood $\mathcal O$ of $\zeta_-$ in such a way that it is analytic in $\mathcal O \setminus \gamma_A$. Namely, by \eqref{otherside}, we should use there
$$
\widehat \phi(z)=\begin{cases}
\widetilde \phi(z), & \text{on the ``$-$'' side of $\sigma_0$}, \\
\widetilde \phi(z)+ \pi i A, & \text{on the ``$+$'' side of $\sigma_0$}, \\
\end{cases}
$$
where again $\sigma_0$ is oriented from  the origin to $\zeta_-$. 
\end{remark}

As a consequence of Theorem~\ref{thm:strongA}, we can make a stronger statement about the zero asymptotics of $p_n$ (compare it with Theorem~\ref{teoLag1}):
\begin{corollary} \label{corollary:zeros}
Under assumptions of Theorem~\ref{thm:strongA}, for every neighborhood $\mathcal O$ of $\gamma_A$, there is $N\in \N$ such that for every $n\geq N$, all zeros of $p_n$ are in $\mathcal O$, and no zeros at the $(\gamma_A)_+$ side of $\gamma_A$.
\end{corollary}
Figure~\ref{fig:trajectoriesWithZeros} is a good illustration of the statement of this corollary.

\begin{figure}[htb]
\centering \begin{overpic}[scale=0.9]%
{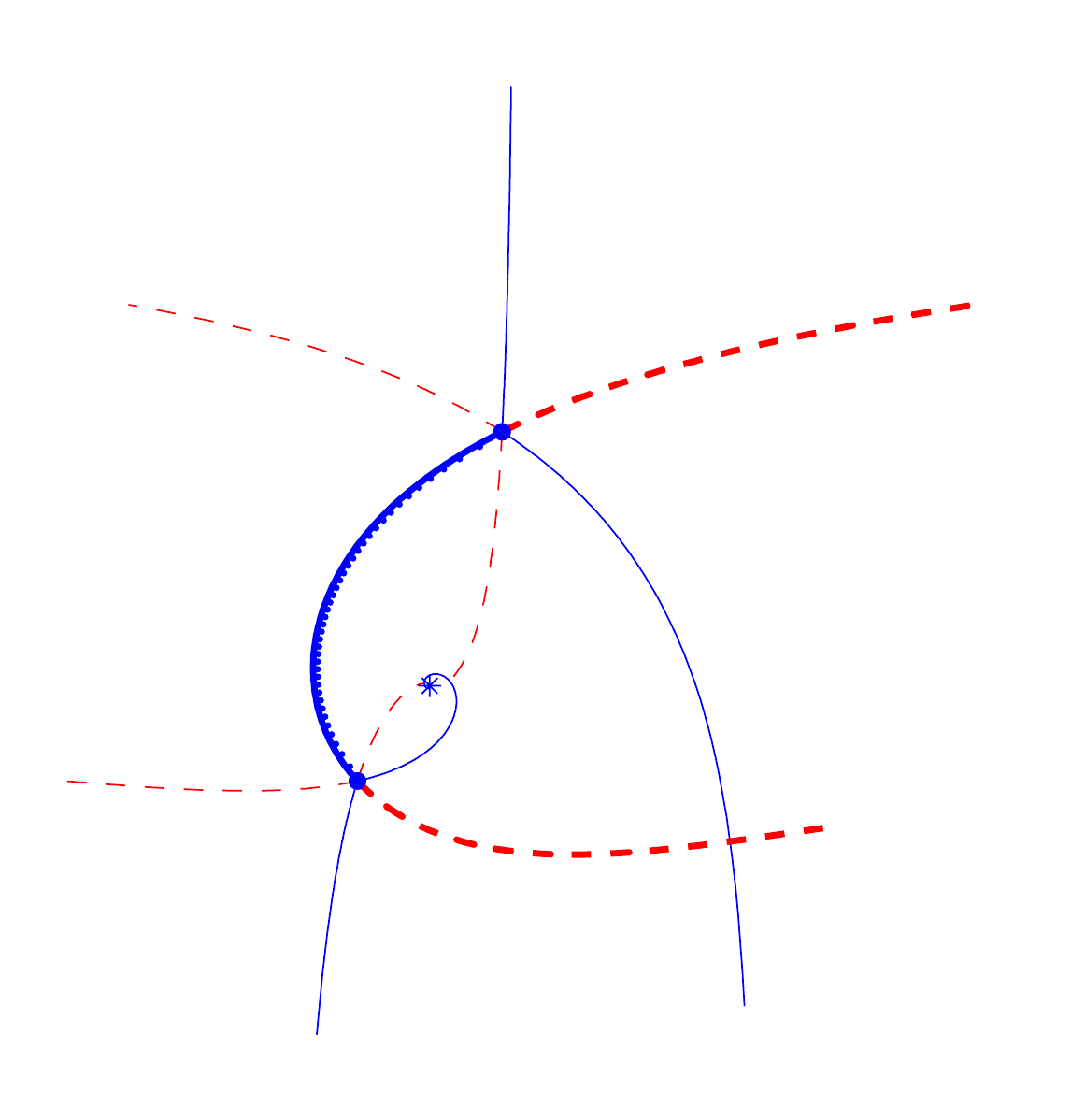}%
    \put(36,37){\scriptsize $0 $}
  \put(28,28){\scriptsize $\zeta_- $}
    \put(43,58){\scriptsize $\zeta_+ $}
     \put(64,71){\small $\Sigma_{ +}$}
          \put(75,25){\small $\Sigma_{-}$}
\put(24,46){\small $\gamma_A$}
\end{overpic}
\caption{Zeros of the Laguerre polynomials (small dots) superimposed to a typical critical graph for the trajectories and orthogonal trajectories of $\varpi_A$.}
\label{fig:trajectoriesWithZeros}
\end{figure}
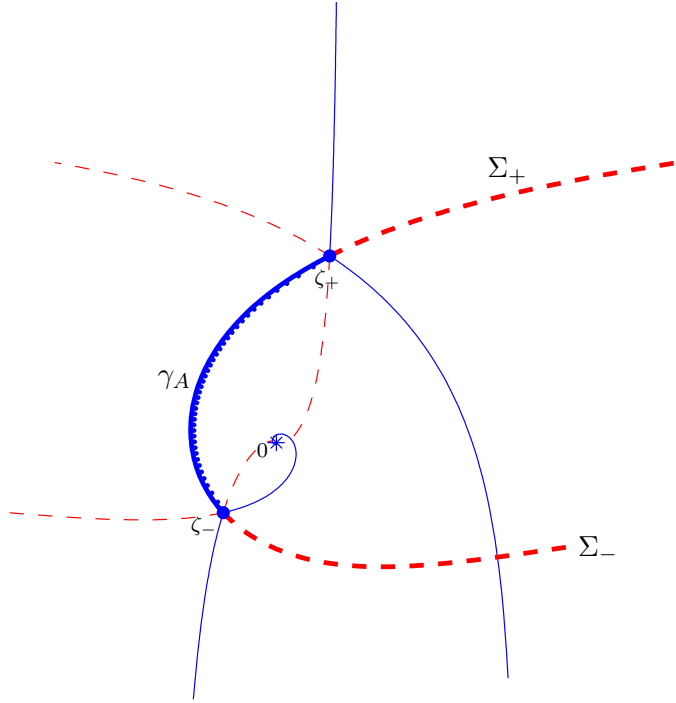

\subsection{Proofs of the asymptotic results} \label{sec:proofs}

All proofs are based on the complex non-hermitian varying
orthogonality satisfied by the generalized Laguerre polynomials with complex coefficients (Theorem \ref{ThOrt}). The cornerstone is the existence of the equilibrium measure $\mu$ having the $S$-property, established in Proposition~\ref{prop:eq_measure}.

The connection of the weak asymptotics of non-hermitian orthogonal polynomials  with the equilibrium with symmetry was established first by Stahl \cite{Stahl:86} (for a fixed weight) and extended by Gonchar and Rakhmanov to a varying orthogonality in  \cite{Gonchar:87}.
Theorem~\ref{teoLag1} is a straightforward consequence of the main theorem in \cite{Gonchar:87}, see also \cite{MR1858305}.

On the other hand, Theorem~\ref{thm:strongA} is established using the non-linear steepest descent analysis of Deif and Zhou \cite{MR99g:34038, MR94d:35143, MR2000g:47048} of the Riemann--Hilbert (RH) problem described in Section~\ref{sec:orthogonality}, where we set $\Sigma = \Sigma_A$, as defined in \eqref{def:SigmaA}. The proof follows the scheme of the work~\cite{Kuijlaars/McLaughlin:01} almost literally. Thus, instead of repeating all the calculations step by step, we describe here the main transformations, referring the interested reader to~\cite{Kuijlaars/McLaughlin:01}  for  details.

\begin{figure}[htb]
\centering
\mbox{\begin{overpic}[scale=0.9]%
{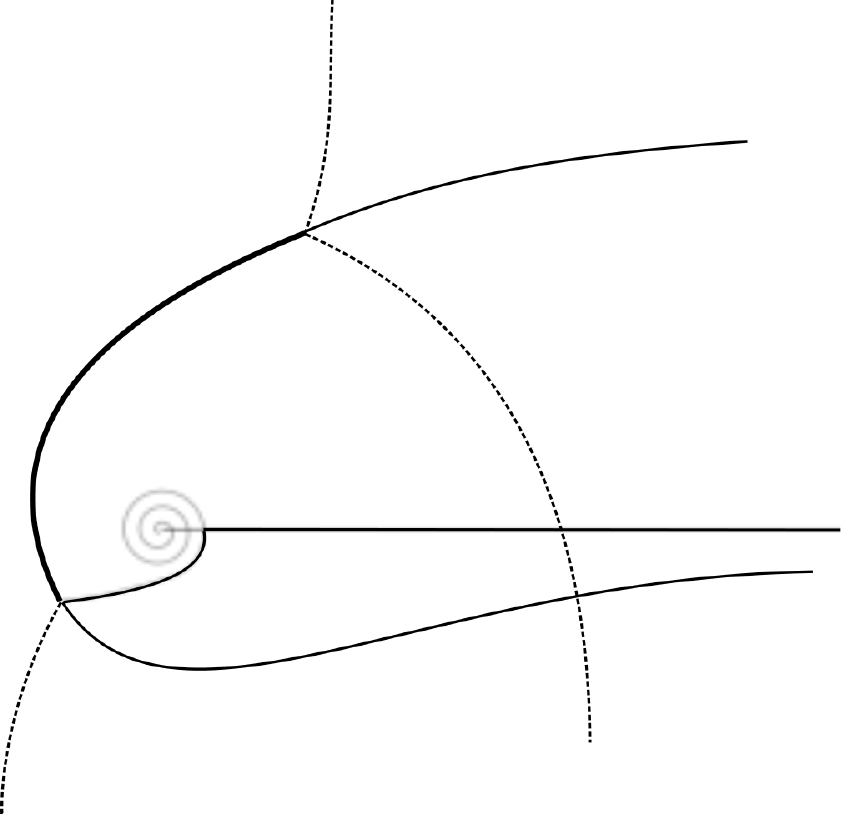}%
       \put(35,65){\small $ \zeta_+ $}
       \put(10,70){\small $ \Omega_+ $}
       \put(30,8){\small $ \Omega_-^{(1)} $}
       \put(65,58){\small $ \Omega_-^{(2)} $}
       \put(6,21){\small $ \zeta_-$}
        \put(20,58){\small $ \gamma_A$}
        \put(2,5){\small $ \sigma_{-}$}
         \put(15,41){\small $ \sigma_{0}$}
        \put(71,10){\small $ \sigma_{\downarrow+}$}
        \put(41,90){\small $ \sigma_{\uparrow+}$}
        \put(85,75){\small $ \Sigma_+$}
         \put(85,24){\small $ \Sigma_-$}
         \put(20,25){\small $ C$}
\end{overpic}} 
\caption{Domain of definition of $\varphi$ in \eqref{newPhiDef}. }
\label{fig:critical_gr_mod}
\end{figure}

It is convenient to redefine slightly function $\phi$ in \eqref{def:phi}, moving the branch cuts from $\sigma_-$ to $\Sigma_-$, and from $\sigma_0$ to $\R_+$ as follows. First, we continue $\phi$ analytically from $\Omega_+$ through $\sigma_-$ to the domain bounded by $\sigma_-$ and $\Sigma_-$. Moreover, let $C$ be the arc of $\sigma_0$ from $\zeta_-$ to its first intersection with $\R_+$, see Figure~\ref{fig:critical_gr_mod}. Then $\Omega_-^{(1)}\setminus (C\cup \R_+)$ consists of two connected components: ``upper'' (containing the origin) and ``lower'' domains. Hence, if we define
\begin{equation}
\label{newPhiDef}
\varphi(z)=\begin{cases}
 \phi(z)- \pi i (2+A), &\text{if $z$ in the domain bounded by $\sigma_-$ and $\Sigma_-$, }\\
 \phi(z)- \pi i A, &\text{if $z$ in the domain bounded by  $\Sigma_-$, $C$ and $\R_+$, } \\
 \phi(z), &\text{if $z$ in the remaining domain bounded by  $\gamma_A$, $C$ and $\R_+$,}
\end{cases}
\end{equation}
it will be holomorphic in $\C\setminus (\gamma_A\cup \Sigma_- \cup \R_+)$, see our analysis in Section~\ref{sec:auxiliary}.

As usual, the first transformation of the Riemann-Hilbert problem for $Y$ is the regularization at infinity by means of the function $g$ introduced in~\eqref{gfunction}. By \eqref{identityV}, there is a constant $\ell$ such that
\begin{equation} \label{main_identity_gfunction}
    g(z) = \frac{1}{2} \left( - A \log z + z - \varphi(z) + \ell \right) ,
    \qquad z \in \mathbb C \setminus (\gamma_A \cup \Sigma_- \cup \R_+),
\end{equation}
where $\log z$ is defined with a branch cut along $\R_+$.

 We define for
$z \in \mathbb C \setminus \Sigma_A$,
\begin{equation*} 
    U(z) = e^{-n (\ell/2) \sigma_3} Y(z)
    e^{-ng(z) \sigma_3} e^{n (\ell/2) \sigma_3}.
\end{equation*}
Here, and in what follows, $\sigma_3$ denotes the Pauli matrix
$\sigma_3 =\left( \begin{array}{cc} 1 & 0 \\ 0 & -1 \end{array} \right)$,
so that for example $e^{-ng(z) \sigma_3} = \left( \begin{array}{cc}
    e^{-ng(z)} & 0 \\ 0 & e^{ng(z)}
    \end{array} \right)$.

    From the Riemann--Hilbert problem for $Y$ it follows by a straightforward
calculation that $U$ is the unique solution of the following RH problem:  determine
$U : \mathbb C \setminus \Sigma \to \mathbb C^{2\times 2}$ such that
\begin{enumerate}
\item[(a)] $U(z)$ is analytic for $z \in \mathbb C \setminus \Sigma_A$,
\item[(b)] $U(z)$ possesses continuous boundary values for $z \in \Sigma_A$,
denoted by $U_+(z)$ and $U_-(z)$,  and
\begin{equation} \label{jumpForU}
    U_+(z) = U_-(z)
    \left( \begin{array}{cc}
    e^{-n(g_+(z) -g_-(z))} & z^{An} e^{-nz} e^{n(g_+(z)+g_-(z) - \ell)} \\
    0 & e^{n(g_+(z) - g_-(z))} \end{array} \right)
\end{equation}
for $z \in \Sigma_A$,
\item[(c)] $U(z)$ behaves like the identity at infinity:
\begin{equation*} 
    U(z) = I +\mathcal O\left(\frac{1}{z}\right)
    \qquad \mbox{ as } z \to \infty, \quad z \in \mathbb C \setminus \Sigma_A.
\end{equation*}
\end{enumerate}

The jump relation \eqref{jumpForU} for $U$ has a different form on
the three parts $\Sigma_A$. Using \eqref{main_identity_gfunction} it is easy to obtain the following jump relations for $g$ across
the contour $\Sigma_A$ (see~\cite{Kuijlaars/McLaughlin:01}):
\begin{equation} \label{eq44}
    U_+(z) = U_-(z)
    \left(\begin{array}{cc}
    e^{ 2n\phi_+(z)} & 1 \\ 0 & e^{ 2n \phi_-(z)}
    \end{array} \right)
    \qquad \mbox{for } z \in \gamma_A,
\end{equation}
\begin{equation} \label{eq45}
    U_+(z) = U_-(z)  \left( \begin{array}{cc}
    1 & e^{- 2n  \phi (z)}  \\
    0 & 1 \end{array} \right)
    \qquad \mbox{ for } z \in \Sigma_+,
\end{equation}
and 
\begin{equation} \label{eq46}
    U_+(z) = U_-(z)  \left( \begin{array}{cc}
    1 & e^{- 2n \widetilde \phi(z)}  \\
    0 & 1 \end{array} \right)
    \qquad \mbox{ for } z \in \Sigma_-.
\end{equation}

The second transformation of the RH problem deals with the oscillatory behavior in the jump matrix for $U$ on $\gamma_A$, see \eqref{eq44}, and is based on its  standard factorization:
\begin{equation} \label{eq47}
    \left(\begin{array}{cc}
    e^{2n\phi_+(z)} & 1 \\ 0 & e^{2n \phi_-(z)}
    \end{array} \right) =
    \left(\begin{array}{cc}
    1 & 0 \\ e^{2n \phi_-(z)} & 1
    \end{array} \right)
    \left(\begin{array}{cc}
    0 & 1 \\ -1 & 0
    \end{array} \right)
    \left(\begin{array}{cc}
    1 & 0 \\ e^{2n \phi_+(z)} & 1
    \end{array} \right).
\end{equation}

As part of the steepest descent method we  introduce the oriented contour $\Sigma^T$, which consists of $\Sigma_A$ plus two simple curves
$\gamma_\pm$  from $\zeta_-$ to $\zeta_+$, contained in $\Omega_+$ and $\Omega_-$,
respectively, as shown in Figure~\ref{fig:lens1}.
We choose $\gamma_\pm$   such that $\Re \phi(z) < 0$ on $\gamma_\pm$, which, as it follows from \eqref{signsPhi}, is always possible. 
Then $\mathbb C \setminus \Sigma^T$ has four connected components, denoted by
$\Omega_1$, $\Omega_2$, $\Omega_3$, and $\Omega_4$ as indicated in Figure~\ref{fig:lens1}.

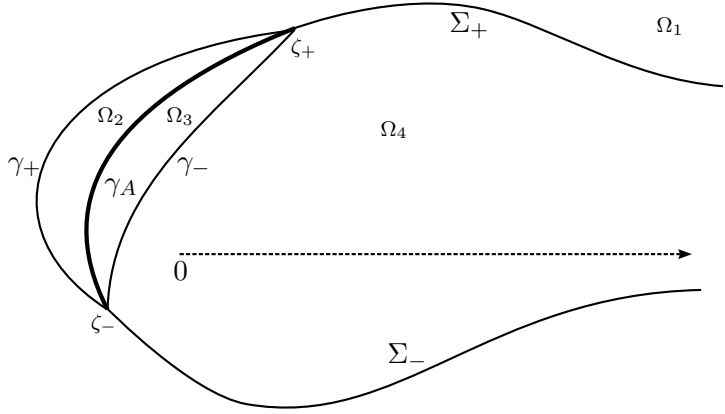
\begin{figure}[htb]
\centering
\begin{overpic}[scale=1]%
{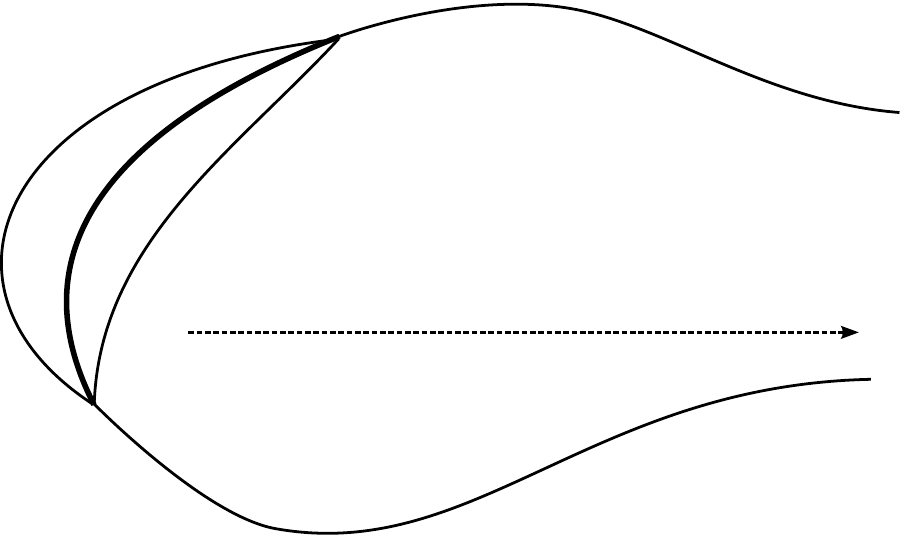}%
 \put(20,19){\small $0 $}
  \put(10,32){\small $\gamma_A $}
  \put(-4,35){\small $\gamma_+ $}
    \put(20.5,35){\small $\gamma_- $}
          \put(60,55){\small $\Sigma_+ $}
      \put(37,52){\scriptsize $\zeta_+ $}
         \put(8,12){\scriptsize $\zeta_- $}
\put(51,7){\small $\Sigma_- $}
\put(90,55){\scriptsize $\Omega_1 $}
\put(9,42){\scriptsize $\Omega_2 $}
\put(19,42){\scriptsize $\Omega_3 $}
\put(50,40){\scriptsize $\Omega_4 $}
\end{overpic}
\caption{Contour $\Sigma^T $ and domains
$\Omega_i$, $i=1, \ldots, 4$.}
\label{fig:lens1}
\end{figure}

Consequently, we define $T : \mathbb C \setminus \Sigma^T \to \mathbb C^{2 \times 2}$ by
\begin{equation*} 
    T(z) = U(z) \qquad \mbox{ for } z \in \Omega_1 \cup \Omega_4,
\end{equation*}
\begin{equation*} 
    T(z) = U(z)  \left(\begin{array}{cc} 1 & 0 \\
    -e^{2n \phi(z)} & 1 \end{array} \right)
    \qquad \mbox{ for } z \in \Omega_2,
\end{equation*}
\begin{equation*} 
    T(z) = U(z)  \left(\begin{array}{cc} 1 & 0 \\
    e^{2n \phi(z)} & 1 \end{array} \right)
    \qquad \mbox{ for } z \in \Omega_3.
\end{equation*}
Then from the RH problem for $U$ and the factorization (\ref{eq47})
we obtain that $T$ is the unique solution of the following Riemann--Hilbert problem:  determine a $2\times 2$ matrix valued function
$T : \mathbb C \setminus \Sigma^T \to \mathbb C^{2\times 2}$ such that
the following hold:
\begin{enumerate}
\item[(a)] $T(z)$ is analytic for $z \in \mathbb C \setminus \Sigma^T$,
\item[(b)] $T(z)$ possesses continuous boundary values for $z \in \Sigma^T$,
denoted by $T_+(z)$ and $T_-(z)$, and
\begin{equation*} 
    T_+(z) = T_-(z)
   \left(\begin{array}{cc} 0 & 1 \\ -1 & 0 \end{array} \right)
   \qquad \mbox{ for } z\in \gamma_A,
\end{equation*}
\begin{equation*} 
    T_+(z) = T_-(z) \left(\begin{array}{cc} 1 & 0 \\
    e^{2n \phi(z)} & 1 \end{array} \right)
    \qquad \mbox{ for } z \in \gamma_\pm,
\end{equation*}
\begin{equation*} 
    T_+(z) = T_-(z) \left( \begin{array}{cc}
     1 & e^{-2 n  \widetilde{\phi}(z)} \\
     0 & 1 \end{array} \right)
     \qquad \mbox{ for } z\in \Sigma_-,
\end{equation*}
and
\begin{equation*} 
    T_+(z) = T_-(z) \left( \begin{array}{cc}
     1 & e^{-2 n \phi(z)} \\
     0 & 1 \end{array} \right)
     \qquad \mbox{ for } z\in \Sigma_+.
\end{equation*}
\item[(c)] $T(z)$ behaves like the identity at infinity:
\begin{equation*} 
    T(z) = I +\mathcal O\left(\frac{1}{z}\right)
    \qquad \mbox{ as } z \to \infty, \quad z \in \mathbb C \setminus \Sigma^T.
\end{equation*}
\end{enumerate}

The global (or outer) parametrix corresponding to this problem is a matrix $N : \mathbb C \setminus \gamma_A \to \mathbb C^{2\times 2}$ given by (see e.g.~\cite{MR99m:42038}, \cite[Section 7.3]{MR2000g:47048}, or \cite[Section 5.1]{Kuijlaars/McLaughlin:01})
\begin{equation*} 
    N(z) =
\begin{pmatrix}
    \left( \frac{1+R'_A(z)}{2}\right)^{1/2} &  -\left( \frac{1-R'_A(z)}{2}\right)^{1/2} \\
     \left( \frac{1-R'_A(z)}{2}\right)^{1/2} &  \left( \frac{1+R'_A(z)}{2}\right)^{1/2}    
    \end{pmatrix},
\end{equation*}
where we take the main branches of the square roots.

It is easy to verify that
\begin{enumerate}
\item[(a)] $N(z)$ is analytic for $z \in \mathbb C \setminus \gamma_A$,
\item[(b)] $N(z)$ possesses continuous boundary values for $z \in
\Gamma \setminus \{\zeta_-, \zeta_+\}$,
denoted by $N_+(z)$ and $N_-(z)$, and
\begin{equation*} 
    N_+(z) = N_-(z)
    \left(\begin{array}{cc} 0 & 1 \\ -1 & 0 \end{array} \right)
    \qquad \mbox{ for } z \in \gamma_A \setminus \{\zeta_-, \zeta_+\},
\end{equation*}
\item[(c)] $N(z) = I + \mathcal O\left(1/z\right)$ for $z \to \infty$.
\item[(d)] Near the endpoints $\zeta_\pm$ it satisfies
$$
N(z)=\mathcal O(|z-\zeta_\pm|^{-1/4})\,, \quad z \to \zeta_\pm.
$$
\end{enumerate}
Since the behavior of $N$ does not match the desired behavior at the endpoints $\zeta_\pm$ of $\gamma_A$, we need one more construction  around these points, namely the so-called local parametrix, well described for instance in \cite{MR2000g:47048}.

From its definition in (\ref{def:phi}) it is easy to see that the $\phi$-function has
a convergent expansion
\begin{equation*} 
    \phi(z) = (z-\zeta_+)^{3/2} \sum_{k=0}^{\infty} c_k (z-\zeta_+)^k,
    \qquad c_0 \neq 0,
\end{equation*}
in a neighborhood of $\zeta_+$. The factor $(z-\zeta_+)^{3/2}$ is defined with a cut
along $\gamma_A \cup \Sigma_-$. Then $f$, defined by
$$
    f(z) = \left[\frac{3}{2} \phi(z) \right]^{2/3},
$$
is  analytic in a neighborhood of $\zeta_+$. We choose the $2/3$-root with
a cut along $\gamma_A$ and such
that $f(z) > 0$ for $z \in \Sigma_+$. Recall that $\Re(\phi) > 0$ on $\Sigma_+$.
Then $f(\zeta_-) = 0$ and $f'(\zeta_-) \neq 0$. Therefore we can deform $\Sigma_+$ locally and choose
$\delta$ so small that $t = f(z)$ is a one-to-one mapping from $U_{\delta}$
onto a convex neighborhood $f(U_{\delta})$ of $t = 0$.
Under the mapping $t = f(z)$, we then have that
$\Sigma_+ \cap U_{\delta}$ corresponds to $(0,+\infty) \cap f(U_{\delta})$
and that $\gamma_A \cap U_{\delta}$ corresponds to $(-\infty,0] \cap f(U_{\delta})$. We can also deform $\gamma_\pm$ in such a way that for an arbitrary, but fixed $\theta \in (\pi/3, \pi)$,   $f$ maps the portion of $\gamma_+$ and $\gamma_-$ in $U_\delta$ to the rays $\{ \arg \zeta = \theta\}$ and $\{ \arg \zeta = -\theta\}$, respectively.

With this mapping we take $P$, analytic for $z \in U_{\delta} \setminus \Sigma^T$,
and continuous on $\overline{U}_{\delta} \setminus \Sigma^T$, given by
\begin{equation*} 
    P(z) = E(z) \Psi^{\theta}(n^{2/3} f(z)) e^{n\phi(z)\sigma_3/2},
\end{equation*}
where
\begin{equation*} 
    E(z) = \sqrt{\pi} e^{\frac{\pi i}{6}} \left(\begin{array}{cc}
        1 & -1 \\ -i & -i \end{array} \right)
        \left(\frac{n^{1/6} f(z)^{1/4}}{a(z)} \right)^{\sigma_3}
\end{equation*}
and $\Psi^{\theta}$ is an explicit matrix valued function built out of
the Airy function $\Ai$ and its derivative $\Ai'$ as follows
\begin{equation*} 
    \Psi^{\theta}(t) = \left\{ \begin{array}{ll}
    \left(\begin{array}{cc}
    \Ai(t) & \Ai(\omega^2 t) \\
    \Ai'(t) & \omega^2 \Ai'(\omega^2 t) \end{array} \right)
    e^{- \frac{\pi i}{6} \sigma_3} &
    \mbox{for } 0 < \arg t < \theta, \\[10pt]
    \left(\begin{array}{cc}
    \Ai(t) & \Ai(\omega^2 t) \\
    \Ai'(t) & \omega^2 \Ai'(\omega^2 t) \end{array} \right)
    e^{- \frac{\pi i}{6} \sigma_3}
    \left(\begin{array}{cc} 1 & 0 \\ -1 & 1 \end{array} \right) &
    \mbox{for } \theta < \arg t < \pi, \\[10pt]
    \left(\begin{array}{cc}
    \Ai(t) & -\omega^2 \Ai(\omega t) \\
    \Ai'(t) & -\Ai'(\omega t) \end{array} \right)
    e^{-\frac{\pi i}{6} \sigma_3}
    \left(\begin{array}{cc} 1 & 0 \\ 1 & 1 \end{array} \right) &
    \mbox{for } -\pi < \arg t < -\theta, \\[10pt]
    \left(\begin{array}{cc}
    \Ai(t) & -\omega^2 \Ai(\omega t) \\
    \Ai'(t) & -\Ai'(\omega t) \end{array} \right)
    e^{- \frac{\pi i}{6} \sigma_3} &
    \mbox{for } -\theta < \arg t < 0,
    \end{array} \right.
\end{equation*}
with $\omega = e^{2\pi i/3}$.

A similar construction yields a parametrix $\tilde{P}$ in a neighborhood
$\tilde{U}_{\delta} = \{ z \mid |z-\zeta_-| < \delta\}$, see \cite{Kuijlaars/McLaughlin:01} for details.

Finally, using $N$, $P$, and $\tilde{P}$, we define for every $n \in \mathbb N$,
\begin{align*}
    S(z) & = T(z) N(z)^{-1} \qquad
    \mbox{for } z \in \mathbb C \setminus
        (\Sigma^T \cup \overline{U}_{\delta} \cup \overline{\tilde{U}}_{\delta}), \\
         S(z) &  = T(z) P(z)^{-1} \qquad \mbox{for } z \in U_{\delta} \setminus \Sigma^T, \\
    S(z) & = T(z) \tilde{P}(z)^{-1} \qquad \mbox{for } z \in \tilde{U}_{\delta} \setminus \Sigma^T.
\end{align*}

Then $S$ is defined and analytic on $\mathbb C \setminus \left(\Sigma^T \cup \partial U_{\delta}
\cup \partial \tilde{U}_{\delta} \right)$. However it follows from the construction that
$S$ has no jumps on $\gamma_A$ and on $\Sigma^T \cap (U_{\delta} \cup \tilde{U}_{\delta})$.
Therefore $S$ has an analytic continuation to $\mathbb C \setminus \Sigma^S$,
where $\Sigma^S$ is the contour indicated in Figure \ref{fig:lens2}. Contour $\Sigma^S$ splits the complex plane in the subdomains also indicated in Figure \ref{fig:lens2}.

\begin{figure}[htb]
\centering
\begin{overpic}[scale=1]%
{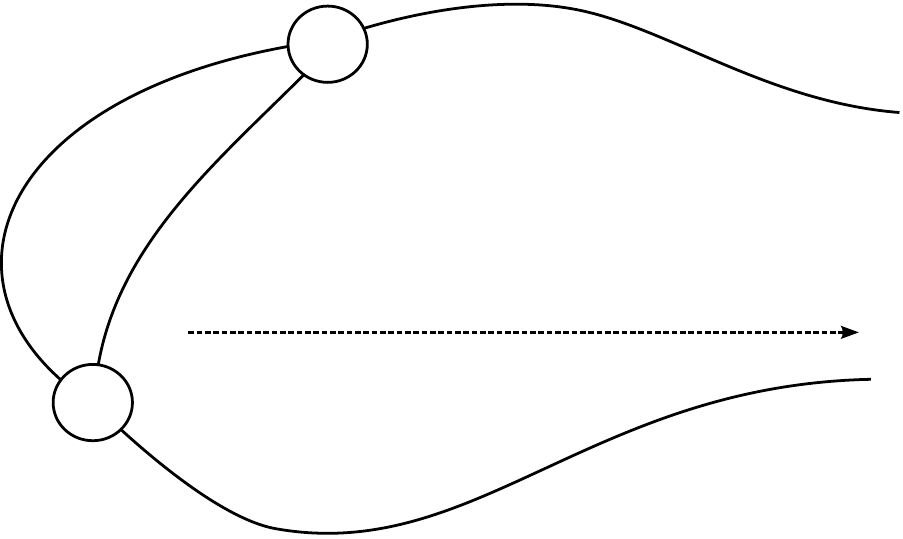}%
 \put(20,19){\small $0 $}
  \put(-4,35){\small $\gamma_+ $}
    \put(20.5,35){\small $\gamma_- $}
          \put(60,55){\small $\Sigma_+ $}
      \put(35,54){\scriptsize $U_{\delta}$}
        \put(8,14){\scriptsize $\tilde{U}_{\delta}$}
\put(51,7){\small $\Sigma_- $}
\put(90,55){\scriptsize $\Omega_1^S $}
\put(15,42){\scriptsize $\Omega_2^S $}
\put(50,40){\scriptsize $\Omega_3^S $}
\end{overpic}
\caption{Contour $\Sigma^S $, and domains $U_{\delta}$, $\tilde{U}_{\delta}$, $\Omega_j^S$,
$j=1,2,3$.}
\label{fig:lens2}
\end{figure}
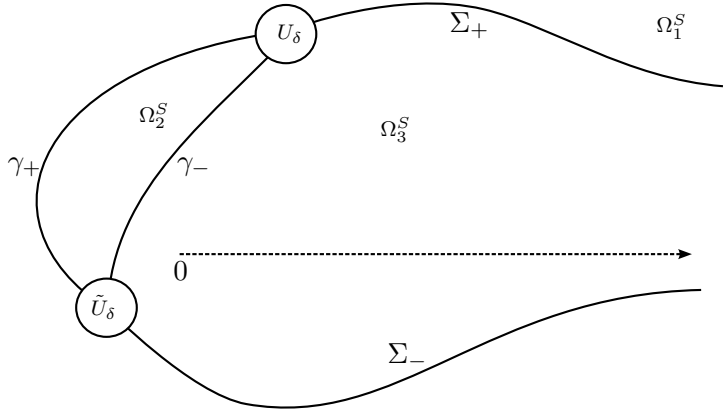

We also have that
$$\| P(z) N^{-1}(z) - I \| \leq \frac{C}{n}\,\,\,\mbox{for }
z \in \partial U_{\delta}\,\,\,\text{and}\,\,\, \| \tilde{P}(z)
N^{-1}(z) - I \| \leq \frac{C}{n}
    \,\,\,\mbox{for } z \in \partial \tilde{U}_{\delta},$$
with a constant $C$ that is independent of $z$ (it can also be chosen
independently of the value of $A$ for $A$ in a compact subset of $\C$,
see \cite{Kuijlaars/McLaughlin:01}), and standard arguments show that jump matrices for $S$ are close to the identity matrix if $n$
is large.

Although in the previous analysis we have used the value  $A$ (assumed fixed), as it was explained in the introduction to Theorem~\ref{thm:strongA}, all of the notions and results introduced
before are $n$-dependent. However, from the asymptotic assumption \eqref{limits} we have that the curves $\gamma_n$ tend to the
limiting curve $\gamma_A$, etc.

We observed already that the jump matrix
for $S_n$ is $I + O(1/n)$ uniformly on $\Sigma_n^S$ as $n \to \infty$.
In addition, the jump matrix converges to the identity matrix as $z \to \infty$
along the unbounded components of $\Sigma_n^S$ sufficiently fast, so that the
jump matrix is also close to $I$ in the $L^2$-sense. Since the contours $\Sigma_n^S$
are only slightly varying with $n$, we may follow standard arguments  to conclude that
\begin{equation} \label{eq72}
    S_n(z) = I + \mathcal O \left(\frac{1}{n}\right) \,\, \mbox{ as } n \to\infty
\end{equation}
uniformly for $z \in \mathbb C \setminus \Sigma_n^S$.

Finally, unraveling  the steps $Y_n \mapsto U_n \mapsto T_n \mapsto S_n$
and using (\ref{eq72}), we obtain strong asymptotics for $Y_n$ in all
regions of the complex plane. In particular we are interested in the
(1,1) entry of $Y_n$, since this is the monic generalized Laguerre
polynomial. We are not describing the details of this straightforward calculation here, and refer again the interested reader to \cite{Kuijlaars/McLaughlin:01}.

This completes the proof of Theorem  \ref{thm:strongA}.

We finish by proving Corollary \ref{corollary:zeros}. The assertion that all zeros of $p_n$ are in $\mathcal O$ (accumulate at $\gamma_A$) is a direct consequence of (a) of Theorem \ref{thm:strongA} and the fact that $1+R'_A(z)\neq 0$ in $\C\setminus \gamma_A$.

Observe that $\Re (R_A')(z)=0$ if and only if $z \in  [\zeta_-,\zeta_+]$, where as in Lemma~\ref{lemma:convexity}, we denote by $ [\zeta_-,\zeta_+]$ the straight segment joining $\zeta_-$ and $\zeta_+$.

By Lemma~\ref{lemma:convexity}, for $A\in \C\setminus [-1,+\infty)$, $\gamma_A\cup [\zeta_-,\zeta_+]$ is a boundary of a simply connected domain; let us denote it by $\mathcal D$. Thus,  $\Re (R_A')$ preserves sign both in $\mathcal D$ and in $\C\setminus \mathcal D$, and these signs are opposite. Since 
$$
\lim_{z\to \infty} R_A'(z)=1,
$$
we conclude that $\Re (R_A')(z)<0$ for $z \in \mathcal D$,  and $\Re (R_A')(z)>0$ for  $z \in \C\setminus \mathcal D$. In particular,
\begin{equation}\label{ineq:R'}
|1+ R_A'(z)| < |1- R_A'(z)| \quad \text{if and only if} \quad z \in \mathcal D.
\end{equation}
Furthermore, for $A<-1$, the inner boundary of $\mathcal D$ corresponds to the ``$-$'' side of $\gamma_A$. Since for $A\in \C\setminus [-1,+\infty)$, $\gamma_A\neq [\zeta_-,\zeta_+]$, we conclude that this property holds for every value of $A\in \C\setminus [-1,+\infty)$. In particular, we have proved that \eqref{ineq:R'} holds only on the ``$-$'' side of $\gamma_A$.

By assertions (b) and (c) of Theorem \ref{thm:strongA}, zeros of $p_n$ in a neighborhood of $\gamma_n$ must satisfy
\[
\left| \frac{1-R_{n}^{\prime }(z)}{1+R_{n}^{\prime }(z)}\right|
^{1/2}e^{2n\Re \phi _{n}(z)}=1+\mathcal O \left(\frac{1}{n}\right). 
\]%
Since by \eqref{signsPhi}, $\Re \phi _{n}(z)< 0$ on both sides of $\gamma$, it remains to use \eqref{ineq:R'} to conclude the proof.

\section*{Acknowledgements}

The  second and the third authors (AMF and PMG) have been supported in part by the research project MTM2011-28952-C02-01 from the Ministry of Science and Innovation of Spain and the European Regional Development Fund (ERDF),  by Junta de Andaluc\'{\i}a, Research Group FQM-229, and by Campus de Excelencia Internacional del Mar (CEIMAR) of the University of Almer\'{\i}a. Additionally, AMF was supported by the Excellence Grant P09-FQM-4643  from Junta de Andaluc\'{\i}a. 
 
MJA and FT have been partially supported  by the research unit UR11ES87 from the University of Gab\`es and the Ministry of Higher
Education and Scientific Research in Tunisia.

The authors acknowledge the contribution of the anonymous referee whose careful reading of the manuscript helped to improve the presentation.

\section*{Bibliography}

%

\def\cprime{$'$}

\end{document}